\documentclass[10pt,twoside,a4paper]{amsart}

\usepackage[usenames,dvipsnames,svgnames,table]{xcolor}
\usepackage{graphicx}
\usepackage[latin1]{inputenc}
\usepackage[OT1]{fontenc}
\usepackage{amsmath}
\usepackage{amsthm}
\usepackage{amssymb}
\usepackage{mathrsfs}
\usepackage[all]{xy}
\xyoption{line}
\newdir{ (}{{}*!/-5pt/@^{(}}
\newcommand{\xycenter}[1]{\begin{center}
                          \mbox{\xymatrix{#1}}
                          \end{center}
                         }
\usepackage{graphicx}
\usepackage{paralist}
\usepackage{overpic} 
\usepackage[backgroundcolor=blue!10]{todonotes} 
\usepackage{booktabs} 
\usepackage[table]{xcolor}
\usepackage{array}
\usepackage{enumitem}
\usepackage{tikz}
\usepackage{threeparttable}
\usepackage{float}
\RequirePackage{doi}
\usepackage{hyperref}
\hypersetup{ 
    colorlinks,
    linkcolor={red!50!black},
    citecolor={blue!50!black},
    urlcolor={blue!80!black}
}

\newtheorem{theorem}{Theorem}[section]

\newtheorem{proposition}[theorem]{Proposition}

\newtheorem{lemma}[theorem]{Lemma}
\newtheorem{corollary}[theorem]{Corollary}

\numberwithin{equation}{section}

\theoremstyle{definition}
\newtheorem{definition}[theorem]{Definition}
\newtheorem{remark}[theorem]{Remark}

\newtheorem*{notation}{Notation}

\newtheorem*{familyCorollary}{Corollary~\ref{cFamilyGeneralizedHexapods}}
\newtheorem*{classificationTheorem}{Theorem~\ref{tClassification}}

\newcommand{\CC}{\mathbb{C}}
\newcommand{\FF}{\mathbb{F}}
\newcommand{\NN}{\mathbb{N}}
\newcommand{\PP}{\mathbb{P}}

\newcommand{\RR}{\mathbb{R}}
\newcommand{\ZZ}{\mathbb{Z}}
\newcommand{\GG}{\mathbb{G}}

\newcommand{\sO}{\mathcal{O}}
\newcommand{\sE}{\mathcal{E}}
\newcommand{\sD}{\mathcal{D}}
\newcommand{\sK}{\mathcal{K}}
\newcommand{\sOPP}{\sO_{\PP^3}}
\newcommand{\sOPii}{\sO_{\PP^1 \times \PP^1}}
\newcommand{\OPn}{\sO_{\PP^n}}

\newcommand{\from}{\leftarrow}

\newcommand{\eulerParam}{\eta}
\newcommand{\eulerProjection}{\pi}

\newcommand{\SO}{\mathrm{SO}}
\newcommand{\SE}{\mathrm{SE}}

\newcommand{\SEbar}{\overline{\SE_3}}
\newcommand{\SEinfty}{\SE_3^\infty}
\newcommand{\SObar}{\overline{\SO_3}}

\newcommand{\PPbondDual}{\widehat{\PP}^{16}}

\newcommand{\EulerInfty}{\PP^3_\infty}

\newcommand{\weightedPP}{\PP(1^4,2^7)}
\newcommand{\weightedEmbedding}{\alpha}

\newcommand{\sOPthree}{\sO_{\PP^3}}
\newcommand{\TPthree}{T_{\PP^3}}

\newcommand{\Spec}{\operatorname{Spec}}

\DeclareMathOperator{\id}{id}
\DeclareMathOperator{\gin}{gin}

\DeclareMathOperator{\initial}{in}
\DeclareMathOperator{\codim}{codim}
\DeclareMathOperator{\rad}{rad}
\DeclareMathOperator{\sat}{sat}

\DeclareMathOperator{\Hilb}{Hilb}

\begin{document}

\title{Hexapods with a small linear span}

\author[Bothmer]{Hans-Christian Graf von Bothmer}
\address{Hans-Christian Graf von Bothmer, Fachbereich Mathematik der Universit\"at Hamburg,
Bundesstra\ss e 55, 20146 Hamburg, Germany}
\email{hans.christian.v.bothmer@uni-hamburg.de}

\author[Gallet]{Matteo Gallet$^\ast$}
\address{Matteo Gallet, International School for Advanced Studies/Scuola Internazionale Superiore di Studi Avanzati (ISAS/SISSA),
Via Bonomea 265, 34136 Trieste, Italy}
\thanks{$^\ast$ Supported by the Austrian Science Fund (FWF): Erwin Schr\"odinger Fellowship J4253.}
\email{mgallet@sissa.it}

\author[Schicho]{Josef Schicho$^\circ$}
\address{Josef Schicho, Research Institute for Symbolic Computation (RISC), Johannes Kepler University Linz,
Altenbergerstra\ss e 69 \\ 
4040 Linz, Austria}
\thanks{$^\circ$ Supported by the Austrian Science Fund (FWF): W1214-N15, project DK9.}
\email{jschicho@risc.jku.at}

\begin{abstract}
The understanding of mobile hexapods, i.e., parallel manipulators with six legs, is one of the driving questions in theoretical kinematics. We aim at contributing to this understanding by employing techniques from algebraic geometry. The set of configurations of a mobile hexapod with one degree of freedom has the structure of a projective curve, which hence has a degree and an embedding dimension. Our main result is a classification of configuration curves of hexapods that satisfy some restrictions on their embedding dimension. 
\end{abstract}

\maketitle

\section{Introduction} \label{sIntro}

In this paper we work on the classification of mechanical manipulators called $n$-pods. 
In particular we focus on the case $n=6$, i.e., \emph{hexapods} or \emph{Stewart-Gough platforms}. 
Geometrically (see~\cite{NawratilIntro}), these are described by $n$ platform anchor points $P_i \in \RR^3$ and $n$ base anchor points $Q_i \in \RR^3$, 
where each pair of base/platform points is connected by a \emph{leg} so that for all possible configurations of the mechanism the distance between~$P_i$ and~$Q_i$ is preserved (see Figure~\ref{fig:hexapod}).

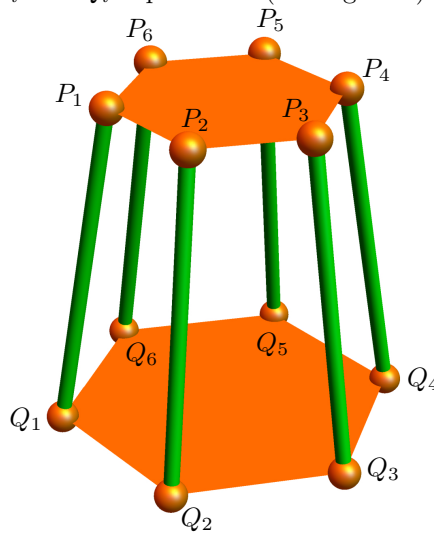
\begin{figure}[ht]
	\begin{overpic}[width=.4\textwidth]{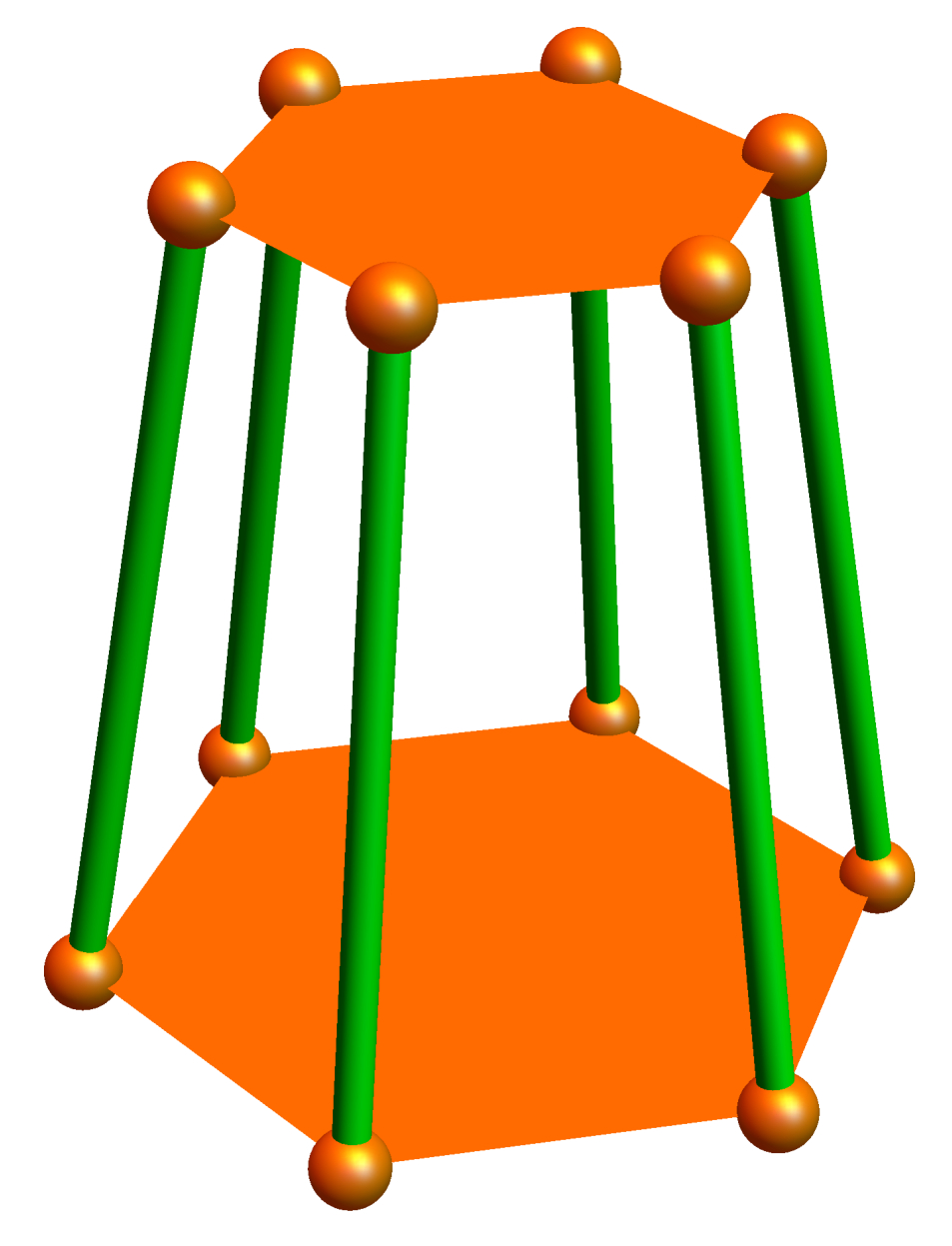}
		\put(-4,20){$Q_1$}
		\put(30,0){$Q_2$}
		\put(67,10){$Q_3$}
		\put(75,28){$Q_4$}
		\put(45,35){$Q_5$}
		\put(19,33){$Q_6$}
		\put(5,85){$P_1$}
		\put(30,80){$P_2$}
		\put(50,82){$P_3$}
		\put(66,90){$P_4$}
		\put(45,100){$P_5$}
		\put(19,98){$P_6$}
	\end{overpic}
	\caption{An example of a hexapod.}
\label{fig:hexapod}
\end{figure}

One is now interested in the \emph{self-motions} of such a configuration, i.e., the set of direct isometries $\sigma \in \SE_3$ of~$\RR^3$ such that
\begin{equation}
\label{eqSpherical}\tag{$\diamondsuit$}
	\left\| \sigma(P_i) - Q_i \right\| = \left\| P_i - Q_i \right\| \quad \text{for all } i \in \{ 1, \dotsc, n \} \,.
\end{equation}
These conditions are called \emph{spherical conditions} or \emph{leg-equations}. 
An $n$-pod is called \emph{mobile}, if the set of its self-motions, 
called the \emph{configuration space} of the $n$-pod, is at least one-dimensional.
One-dimensional subvarieties of this set are called \emph{configuration curves}.

The group of direct isometries~$\SE_3$ of~$\RR^3$ is $6$-dimensional, 
since it is the semidirect product of the $3$-dimensional group of rotations with the $3$-dimensional group of translations. 
Therefore, an $n$-pod is always expected to move if $n<6$ since in this case there are at most $5$ spherical conditions. 
The most interesting case is then $n=6$, which is the smallest value for which a generic $n$-pod does not move, but special ones do. 
The classification of mobile hexapods is still wide open. 

Consider the spherical conditions from Equation~\eqref{eqSpherical}. 
These define linearly equivalent divisors on~$\SE_3$. 
The linear system containing these divisors defines a natural embedding
\[
	\SE_3 \subset \PP^{16}
\]
as was noticed by~\cite{BondTheory} and \cite[Section~5]{Mourrain1996}. 
Let $\SEbar$ be the compactification of~$\SE_3$ in this embedding: 
it is a projective variety of dimension~$6$ and degree~$40$. 
By construction, the spherical conditions are now \emph{linear} in these coordinates, 
i.e., they define hyperplanes in~$\PP^{16}$. 

In this paper we take the view that the basic invariants of a mobile hexapod with one degree of freedom
are the degree and the codimension of the span of its configuration curve in~$\PP^{16}$. 
In this setting, we have the following restrictions, which we depict in Table~\ref{tPhaseDiagram}:
\begin{enumerate}[label=(R\Alph*)]
 \item No configuration curve can have degree higher than~$40$ because they are linear sections of~$\SEbar$, which itself has degree~$40$. 
 Moreover, since $\SEbar$ has dimension~$6$, then no configuration curve can have a span of codimension smaller than~$5$. \label{restriction:degree}
 \item Since $\SEbar \subset \PP^{16}$ is defined by quadrics, if the dimension of the span of a configuration curve is~$i$ then the degree of that curve cannot be more than~$2^{i-1}$. \label{restriction:quadrics}
 \item From general results about varieties of minimal degree, we know that the dimension of the span of a configuration curve cannot exceed its degree, with equality for rational normal curves. \label{restriction:minimal}
\end{enumerate}
Moreover, in the literature one can find the following constructions and restrictions, also reported in Table~\ref{tPhaseDiagram}:
\begin{enumerate}[label=(L\arabic*)]
 \item When the codimension of the span is precisely~$5$, then the configuration curve has degree~$40$.
 All these hexapods are \emph{architecturally singular}. 
 These are pods whose legs in a prescribed configurations can be perturbed with less than $6$ degrees of freedom (which is the standard situation). 
 This is a well-studied class of mechanisms; see \cite{Husty2000a, Roschel1998, Karger2008, Nawratil2009}. 
 In \cite{Hauenstein2018}, where investigations about architecturally singular hexapods are made using numerical algebraic geometry, 
 a subfamily whose elements are called \emph{Segre-dependent Stewart-Gough platforms} is studied. \label{literature:architecturally_singular}
 \item From \cite{Gallet2017}, we know that the values $(28, 6)$ are attained at the family of so-called \emph{liaison hexapods}; moreover, $28$ is also the maximal possible degree for hexapods that are not architecturally singular and do not satisfy some degeneracy condition, and it is attained only for the liaison case. We do not have information about hexapods satisfying these degeneracy condition. \label{literature:liaison}
 \item The value $(12, 10)$ is attained by hexapods obtained by taking only $6$ legs from a mobile icosapod (which is a pod with $20$ legs); see~\cite{Borel1908} and \cite{GalletNawratilSchichoSelig} for a more recent account. To this class belong the hexapods created by Gei\ss\ and Schreyer~\cite{Geis2009} using computations over finite fields. \label{literature:icosapod}
 \item The values $(6,10)$, $(8,9)$, and $(12,10)$ are attained by \emph{Bricard flexible octahedra}, where three pairs of base and platform points coincide; see~\cite{Bricard1897}.\label{literature:octahedra} 
 \item The value $(4,13)$ is attained by the so-called Borel-Bricard hexapods, whose configuration curve is a quartic of genus~$1$. The two foundational papers of Borel~\cite{Borel1908} and Bricard~\cite{Bricard1906} describe these mobile hexapods. \label{literature:borel_bricard}
 \item The value $(2,14)$ is attained by hexapods allowing a so-called \emph{butterfly motion}, where all platform points move around a fixed axis, and so the configuration curve is a conic; see~\cite{Nawratil2014a}. \label{literature:butterfly}
 \item The value $(24,6)$ is attained for \emph{(orthogonal) point-symmetric} hexapods; see~\cite{Nawratil2014a} and \cite{Dietmaier1996}. \label{literature:point_symmetric}
 \item Nawratil investigated mobile hexapods in which base and platform differ by a projectivity~\cite{Nawratil2012}, or an isometry~\cite{Nawratil2014c}, or an isometry followed by a scaling~\cite{Nawratil2013c}. The latter two cases provide the values~$(4,13)$ and~$(18,8)$, respectively. \label{literature:congruent_equiform}
\end{enumerate}
We thank Georg Nawratil for providing us information about these known families of mobile hexapods, 
and we refer to his recent work~\cite{Nawratil2018} for a more detailed account on these families.

\definecolor{orange}{RGB}{230,159,0}
\definecolor{skyblue}{RGB}{86,180,233}
\definecolor{bluishgreen}{RGB}{0,158,115}
\definecolor{vermilion}{RGB}{213,94,0}
\definecolor{reddishpurple}{RGB}{204,121,167}
\definecolor{yellow}{RGB}{240,228,66}

\newcommand{\ccol}{\cellcolor}
\newcommand{\rectcol}[3]{\tikz \fill [#1] (0,0) rectangle (#2,#3);}
\renewcommand{\arraystretch}{1.3}

\begin{center}
\begin{table}[ht]
 \begin{threeparttable}
  \caption{Degrees and codimension (in $\PP^{16}$) of the span of configuration curves of mobile hexapods.}
    \begin{tabular}{ c | p{0.06\textwidth} p{0.06\textwidth} p{0.06\textwidth} p{0.06\textwidth} p{0.06\textwidth} p{0.06\textwidth} p{0.06\textwidth} p{0.06\textwidth} p{0.06\textwidth} p{0.06\textwidth} }
 \toprule
 40 & \ccol{orange} & \ccol{orange}& \ccol{orange} & \ccol{orange} & \ccol{orange} & \ccol{orange} & \ccol{orange} & \ccol{orange} & \ccol{orange} & \ref{literature:architecturally_singular} \\ \hline
 $\vdots$ & \ccol{skyblue} & \ccol{skyblue} & \ccol{skyblue} & \ccol{skyblue} & \ccol{skyblue} & \ccol{bluishgreen} & \ccol{bluishgreen} & \ccol{bluishgreen} & \ccol{bluishgreen} & \ccol{orange}\\ \hline
 32 & \ccol{skyblue} & \ccol{skyblue} & \ccol{skyblue} & \ccol{skyblue} & \ccol{skyblue} & \ccol{bluishgreen} & \ccol{bluishgreen} & \ccol{bluishgreen} & \ccol{bluishgreen} & \ccol{orange}\\ \hline
 $\vdots$ & \ccol{skyblue} & \ccol{skyblue} & \ccol{skyblue} & \ccol{skyblue} & \ccol{bluishgreen} & \ccol{bluishgreen} & \ccol{bluishgreen} & \ccol{bluishgreen} & \ccol{bluishgreen} & \ccol{orange}\\ \hline
 28 & \ccol{skyblue} & \ccol{skyblue} & \ccol{skyblue} & \ccol{skyblue} & \ccol{bluishgreen} & \ccol{bluishgreen} & \ccol{bluishgreen} & \ccol{bluishgreen} & \ref{literature:liaison} & \ccol{orange}\\ \hline
 $\vdots$ & \ccol{skyblue} & \ccol{skyblue} & \ccol{skyblue} & \ccol{skyblue} & \ccol{vermilion} & \ccol{vermilion} & \ccol{vermilion} & \ccol{vermilion} && \ccol{orange}\\ \hline
 24 & \ccol{skyblue} & \ccol{skyblue} & \ccol{skyblue} & \ccol{skyblue} & \ccol{vermilion} & \ccol{vermilion} & \ccol{vermilion} & \ccol{vermilion} & \ref{literature:point_symmetric} & \ccol{orange}\\ \hline
 $\vdots$ & \ccol{skyblue} & \ccol{skyblue} & \ccol{skyblue} & \ccol{skyblue} & \ccol{vermilion} & \ccol{vermilion} & \ccol{vermilion} & \ccol{vermilion} && \ccol{orange}\\ \hline
 18 & \ccol{skyblue} & \ccol{skyblue} & \ccol{skyblue} & \ccol{skyblue} & \ccol{vermilion} & \ccol{vermilion} & \ref{literature:congruent_equiform} \ccol{vermilion} & \ccol{vermilion} && \ccol{orange}\\ \hline
 16 & \ccol{skyblue} & \ccol{skyblue} & \ccol{skyblue} & \ccol{skyblue} & \ccol{vermilion} & \ccol{vermilion} & \ccol{reddishpurple} & \ccol{vermilion} && \ccol{orange}\\ \hline
 14 & \ccol{skyblue} & \ccol{skyblue} & \ccol{skyblue} && \ccol{vermilion} & \ccol{vermilion} & \ccol{reddishpurple} & \ccol{reddishpurple} && \ccol{orange}\\ \hline
 12 & \ccol{skyblue} & \ccol{skyblue} & \ccol{skyblue} & \ref{literature:icosapod} & \ref{literature:octahedra} \ccol{vermilion} & \ccol{vermilion} & \ccol{reddishpurple} & \ccol{reddishpurple} && \ccol{orange}\\ \hline
 10 & \ccol{skyblue} & \ccol{skyblue} & \ccol{skyblue} && \ccol{vermilion} & \ccol{vermilion} & \ccol{reddishpurple} & \ccol{reddishpurple} && \ccol{orange}\\ \hline
 8  & \ccol{skyblue} & \ccol{skyblue} &&& \ccol{vermilion} & \ref{literature:octahedra} \ccol{reddishpurple} & \ccol{reddishpurple} & \ccol{yellow} & \ccol{yellow} & \ccol{orange}\\ \hline
 6  & \ccol{skyblue} & \ccol{skyblue} &&& \ref{literature:octahedra} \ccol{reddishpurple} & \ccol{yellow} & \ccol{yellow} & \ccol{yellow} & \ccol{yellow} & \ccol{orange}\\ \hline
 4  & \ccol{skyblue} & \ref{literature:borel_bricard} \ref{literature:congruent_equiform} && \ccol{yellow} & \ccol{yellow} & \ccol{yellow} & \ccol{yellow} & \ccol{yellow} & \ccol{yellow} & \ccol{orange}\\ \hline
 2  & \ref{literature:butterfly} & \ccol{yellow} & \ccol{yellow} & \ccol{yellow} & \ccol{yellow} & \ccol{yellow} & \ccol{yellow} & \ccol{yellow} & \ccol{yellow} & \ccol{orange}\\ \hline
   & 14 & 13 & 12 & 11 & 10 & 9 & 8 & 7 & 6 & 5 \\
 \bottomrule
\end{tabular}
 \begin{tablenotes}[para,flushleft]
{\protect\rectcol{orange}{0.6}{0.25}} These entries cannot occur because of restriction \ref{restriction:degree}. \\
{\protect\rectcol{skyblue}{0.6}{0.25}} These entries cannot occur because of restriction \ref{restriction:quadrics}. \\
{\protect\rectcol{yellow}{0.6}{0.25}} These entries cannot occur because of restriction \ref{restriction:minimal}. \\
{\protect\rectcol{bluishgreen}{0.6}{0.25}} If these entries occur, by item \ref{literature:liaison} they must satisfy some degeneracy conditions. \\
{\protect\rectcol{vermilion}{0.6}{0.25}} For these hexapods, the span of their configuration curve must intersect the center of the projection $\hat{\pi} \colon \PP^{16} \dashrightarrow \PP^9$. This is proved in this paper. \\
{\protect\rectcol{reddishpurple}{0.6}{0.25}} These mobile hexapods exist (over the complex numbers). This is proved in this paper.
\end{tablenotes}
\end{threeparttable}
\label{tPhaseDiagram}
\end{table}
\end{center}
\renewcommand{\arraystretch}{1}

To contribute to this classification, we consider $\PPbondDual$, the space of all hyperplanes in~$\PP^{16}$. 
In this space, the possible spherical conditions (recall Equation~\eqref{eqSpherical}) for pairs of points~$(P,Q)$ form a subvariety
\[
	\RR^3 \times \RR^3 \subset \PPbondDual \,.
\]
Compactifying this space, we obtain
\[
	\PP^3 \times \PP^3 \subset \PPbondDual 
\]
where the inclusion is given by the Segre embedding. The classification of mobile hexapods is then the following task:

\begin{center}
``Find $6$ points in $\PP^3 \times \PP^3 \subset \PPbondDual$ such that the intersection of the corresponding hyperplanes with~$\SEbar$ contains a curve~$D$.''
\end{center}

In this paper we adopt a slight, but useful, change of perspective:

\begin{center}
``Find curves $D \subset \SEbar$ such that the space $\langle D \rangle^\perp \subset \PPbondDual$ of hyperplanes containing~$D$ intersects~$\PP^3 \times \PP^3$ in at least $6$ points.''
\end{center}

The advantage of this perspective is that the classification of algebraic curves~$D$ inside a larger algebraic variety is a classical and well-developed branch of algebraic geometry, which provides the tools for the present article. 
In particular, the condition that the ideal of~$D$ contains many linear forms gives a strong restriction on the possible types of such curves. 

To state our main theorems, some further (well-known) geometry is needed. Consider the natural projection
\[
	\pi \colon \SE_3 \to \SO_3 \,.
\]
By a theorem of Euler we can represent~$\SO_3$ as an open subset of~$\PP^3$ (see Theorem~\ref{tEuler}). 
The natural projection above then extends to a rational map
\[
	\pi \colon \SEbar \dashrightarrow \PP^3 \,,
\]
which, in turn, is the composition of a linear projection $\hat{\pi} \colon \PP^{16} \dashrightarrow \PP^9$ 
and the inverse of the Veronese embedding $\PP^3 \longrightarrow \PP^9$.
We say that the projection under~$\pi$ of a curve in~$\SEbar$ is an \emph{Euler curve}.

In this paper we prove the following:

\begin{classificationTheorem} 
Let $D \subset \SEbar \subset \PP^{16}$ be an irreducible curve such that 
the span~$\langle D \rangle$ does not intersect the center of the projection $\hat{\pi} \colon \PP^{16} \dashrightarrow \PP^9$ defined above. 
Assume that $\codim \langle D \rangle = c+7$ for some nonnegative number~$c$. 
Then $\pi(D)$ is one of the curves listed in Table~\ref{tWithQuadrics} and Table~\ref{tWithoutQuadrics}.
\end{classificationTheorem}

In particular, this shows that hexapods with invariants in Table~\ref{tPhaseDiagram} colored in vermilion (\rectcol{vermilion}{0.6}{0.25}) cannot satisfy the hypotheses of Theorem~\ref{tClassification}. 

We call a curve $D \subset \SEbar$ defined over~$\CC$ a \emph{generalized $n$-pod curve} if 
\[
	S(D) := \langle D \rangle^\perp \cap (\PP^3 \times \PP^3) \subset \PPbondDual
\]
is a zero-dimensional scheme of length at least~$n$. 
Every curve in~$\SEbar$ coming from a mobile hexapod is a generalized hexapod curve, where the leg-equations correspond to the points of~$S(D)$. 
Conversely, a generalized hexapod curve~$D$ comes from a mobile hexapod if the following conditions hold:
\begin{enumerate}
\item $S(D)$ is reduced;
\item $S(D)$ lies in the affine part of~$\PP^3 \times \PP^3$;
\item the reduced points of~$S(D)$ are individually defined over~$\RR$.
\end{enumerate}

The second main result of this paper is the following:

\begin{familyCorollary} 
For each entry in Table~\ref{tWithQuadrics} and Table~\ref{tWithoutQuadrics} not marked by $(\ast)$, 
there exists a positive-dimensional family of generalized hexapod curves with Euler curves of those invariants.
Lower bounds for the dimensions of these families are listed in Table~\ref{tWithQuadricsDimensions} and Table~\ref{tWithoutQuadricsDimensions}.
\end{familyCorollary}

The main new tool introduced in this paper, apart from the use of results of classical algebraic geometry, 
is the observation that there is a natural embedding of~$\SEbar$ in weighted projective space~$\weightedPP$ compatible with all the geometry described above. 
In this embedding, the variety~$\SEbar$ is an arithmetically Gorenstein $6$-fold of codimension~$4$ and therefore has highly structured equations. 
More precisely, the ideal of~$\SEbar$ in this embedding is generated by the $4 \times 4$ Pfaffians of a skew-symmetric $5 \times 5$ matrix and an additional equation of weighted degree~$4$.

The paper is structured as follows. 
Section~\ref{sCompactification} defines the compactification~$\SEbar$ of~$\SE_3$ and introduces the embedding $\SEbar \subset \weightedPP$. 
Section~\ref{sHexapods} describes $n$-pods and $n$-pod curves; this makes our change of perspective described above precise. 
Section~\ref{sInfinity} discusses the behavior at infinity of~$\SEbar$ and $\SObar \cong \PP^3$. 
In particular, we observe that $\SEbar$ is contact to the hyperplane at infinity and we clarify the relation between the contact condition in~$\PP^3$ and the existence of limits of hexapod configurations called butterfly bonds. 
The main technical tools are developed in Section~\ref{sPfaffianEquations} where we deduce some geometric consequences of the Pfaffian equations derived in Section~\ref{sCompactification}. 
This is used in Section~\ref{sClassification} to prove our Classification Theorem~\ref{tClassification}. 
In Section~\ref{sConstruction} we prove the Construction Theorem~\ref{tFamilies} and Corollary~\ref{cFamilyGeneralizedHexapods}. 
The Appendices collect useful facts from the classification of space curves and finite fields techniques.

Some of our results rely on computer algebra calculations.
The reader can find the relevant code files on Zenodo at~\cite{software}.

\textbf{Acknowledgments.} We thank the Erwin Schr\"{o}dinger Institute (ESI) of the University of Vienna for the hospitality during the workshop ``Rigidity and Flexibility of Geometric Structures'', when the three authors could meet and discuss. Moreover, we thank Georg Nawratil for useful discussions about known families of mobile hexapods.

\section{Compactifications of \texorpdfstring{$\SE_3$}{SE3}} \label{sCompactification}

In this section we discuss a compactification of~$\SE_3$ 
introduced by Gallet, Nawratil and Schicho in~\cite{BondTheory}.

\begin{definition}
Let $\SE_3$ be the algebraic group of direct isometries of~$\RR^3$ with respect to the standard scalar product, 
i.e., the set of pairs $(A,y)$ with $A \in \SO_3$ and $y \in \RR^3$. 
The action of~$\SE_3$ on~$\RR^3$ is then given by
\[
	v \mapsto Av + y \,.
\]
\end{definition}

The compactification of~$\SE_3$ defined in~\cite{BondTheory} is the following:

\begin{definition} \label{dBondEmbedding}
Consider the injective morphism
\begin{align*}
	\phi \colon \SE_3 &\longrightarrow \PP^{16}_{\RR} \\
	(A,y) &\mapsto \bigl(A:\underbrace{-A^t y}_{:=x}:y:\underbrace{\langle y,y \rangle}_{:=r}:1\bigr)
\end{align*}
The coordinate ring of~$\PP^{16}_{\RR}$ is $R := \RR[a_{1,1}, \dotsc, a_{3,3}, x_1, \dotsc, x_3, y_1, \dotsc, y_3, r, h]$.

Let now $I_{\SEbar} \subset R$ be the ideal of homogeneous polynomials vanishing on~$\phi(\SE_3)$ and set
\[
	\SEbar := V(I_{\SEbar})
\]
the vanishing set of these polynomials. 
The variety $\SEbar \subset \PP^{16}_{\RR}$ is then a compactification of~$\SE_3$.
\end{definition}

\begin{remark}
\cite{BondTheory} gives explicit equations of~$\SEbar$.
\end{remark}

The important advantage of this embedding is that the spherical conditions are \emph{linear} in the above coordinates:

\begin{proposition}[Gallet-Nawratil-Schicho] \label{pLegEquations}
Let $P = (p_1,p_2,p_3)$ and $Q = (q_1,q_2,q_3)$ be two points in~$\RR^3$, and $(A,x,y,r,h) \in \SE_3$.
Then the condition $\left\| (A,y)(P) - Q \right\| = \left\| P-Q \right\|$ can be written as
\[
	\begin{pmatrix} Q^t & 1 \end{pmatrix}
	\begin{pmatrix}
		A - h \id & y \\
		x^t & - \frac{1}{2} r
	\end{pmatrix}
	\begin{pmatrix} P \\ 1 \end{pmatrix} 
	 = 0.
\]
This equation is called the \emph{leg-equation} $l_{P,Q}$ of~$P$ and $Q$.
\end{proposition} 

\begin{proof}
Using the scalar product in~$\RR^3$, we can rewrite the equation above as
\begin{align*}
	0 = &\, \langle AP + y - Q, AP + y - Q \rangle - \langle P-Q,P-Q \rangle \\
	  = &\, \langle AP,AP \rangle + 2 \langle y,AP \rangle - 2 \langle Q,AP \rangle - 2 \langle Q,y \rangle + \langle Q,Q \rangle + \langle y,y \rangle \\
	    & - (\langle P,P \rangle + \langle Q,Q \rangle - 2 \langle Q,P \rangle) \\
	  = &\, 2 \langle y,AP \rangle - 2 \langle Q,AP \rangle - 2 \langle Q, y\rangle + \langle y,y \rangle + 2 \langle Q,P \rangle \\
	  = &\, -2 \langle x,P \rangle - 2 \langle Q,AP \rangle - 2 \langle Q,y \rangle + r + 2 \langle Q,P \rangle \\
	  = &\, -2 \bigl(\langle x,P \rangle + \langle Q,AP \rangle + \langle Q,y \rangle - \frac{1}{2} r - \langle Q,P \rangle h\bigr)
\end{align*}
with~$h=1$. But this is exactly what we get if we expand the matrix equation in the statement.
\end{proof}

We now recall a surprising theorem of Euler:

\begin{theorem}[Euler]
\label{tEuler}
Let $\SO_3$ be the group of orthogonal $3 \times 3$ matrices. 
Let $\PP^3_{\RR}$ be the projective space with coordinates $e_0, \dotsc, e_3$. 
Set furthermore
\[
	h := e_0^2+e_1^2 + e_2^2 + e_3^2
\]
and let 
$\bigl(\PP^3_{\RR}\bigr)^0 := \PP^3_{\RR} \backslash \{ h=0 \}$ 
be the open set where $h$ does not vanish. 
Then the morphism
\[
	\eulerParam \colon \bigl( \PP^3_{\RR}\bigr)^0 \to \SO_3
\]
given by 
\[
	\eulerParam(e_0:e_1:e_2:e_3) := 
	\frac{1}{h}
	\begin{pmatrix}
		e_{0}^{2}+e_{1}^{2}-e_{2}^{2}-e_{3}^{2}&
		2 e_{1} e_{2} - 2 e_{0} e_{3}&
		2 e_{0} e_{2} + 2 e_{1} e_{3}\\
		2 e_{1} e_{2} + 2 e_{0} e_{3}&
		e_{0}^{2}-e_{1}^{2}+e_{2}^{2}-e_{3}^{2}&
		-2 e_{0} e_{1} + 2 e_{2} e_{3}\\
		-2 e_{0} e_{2} + 2 e_{1} e_{3}&
		2 e_{0} e_{1} + 2 e_{2} e_{3}&
		e_{0}^{2}-e_{1}^{2}-e_{2}^{2}+e_{3}^{2}
	\end{pmatrix}
\]
is an isomorphism. 
\end{theorem}

\begin{proof}
This is a straightforward computation.
\end{proof}

Using this, we construct a new embedding for~$\SEbar$:

\begin{definition} \label{dWeightedEmbedding}
Euler's parametrization~$\eulerParam$ induces an embedding
\[
	\weightedEmbedding \colon \weightedPP \to \PP^{16}
\]
with
\begin{align*}
	&\weightedEmbedding(e_0:e_1:e_2:e_3:x_1:x_2:x_3:y_1:y_2:y_3:r) \\
	&:= \bigl( h \, \eulerParam(e_0:e_1:e_2:e_3) : x_1:x_2:x_3:y_1:y_2:y_3:h:r \bigr)
\end{align*}
where recall that $h = e_0^2 + e_1^2 + e_2^2 + e_3^2$.
\end{definition}

\begin{remark}
This is just the $2$-uple embedding of~$\weightedPP$ written in interesting coordinates.
\end{remark}

\begin{proposition} \label{pEquationsWeightedPP}
There is a natural embedding
\[
	\SEbar \subset \weightedPP
\]
compatible with $\weightedEmbedding$ such that the ideal of~$\SEbar$ is
minimally generated by the $4 \times 4$ Pfaffians of
\[
	M =
	\begin{pmatrix}
		0&
		-x_{1}+y_{1}&
		-x_{2}+y_{2}&
		-x_{3}+y_{3}&
		{e_{0}}\\
		x_{1}-y_{1}&
		0&
		x_{3}+y_{3}&
		-x_{2}-y_{2}&
		{e_{1}}\\
		x_{2}-y_{2}&
		-x_{3}-y_{3}&
		0&
		x_{1}+y_{1}&
		{e_{2}}\\
		x_{3}-y_{3}&
		x_{2}+y_{2}&
		-x_{1}-y_{1}&
		0&
		{e_{3}}\\
		-e_{0}&
		-e_{1}&
		-e_{2}&
		-e_{3}&
0
	\end{pmatrix} 
\]
and the quartic equation
\[
 y_1^2+y_2^2+y_3^2 - rh = 0.
\]
Furthermore, the identity element of~$\SE_3$ has coordinates $(1:0:\dots:0)$ in this embedding.
\end{proposition}

\begin{proof}
The image of~$\weightedEmbedding$ contains~$\SEbar$ since the image of~$\eulerParam$ is 
the set of orthogonal matrices. We obtain the generators of the ideal of $\SEbar \subset \weightedPP$
by pulling back the generators of $\SEbar \subset \PP^{16}$ and saturating the resulting ideal. 
This can be calculated by a computer algebra system.
\end{proof}

\begin{remark}
Notice that the above equations have integer coefficients. They therefore make sense over any field~$K$. 
Over such a field~$K$ we denote by~$\SEbar$ the variety defined by the above equations.
\end{remark}

\begin{definition}
Let $\eulerProjection \colon \weightedPP \dashrightarrow \PP^3$ be the rational morphism 
obtained by projecting from the codimension~$4$ subspace $\{ e_0 = e_1 = e_2 = e_3 = 0 \} \subset \weightedPP$. 
We denote the restriction of this rational morphism to~$\SEbar$ by the same letter.
\end{definition}

\section{Hexapods} \label{sHexapods}

\noindent
Let us formally define what a mobile $n$-pod is:

\begin{definition}
An \emph{$n$-pod} is a pair of $n$-tuples $(P_1,\dotsc,P_n)$ and $(Q_1,\dots,Q_n)$ in~$\RR^3$. 
Consider the corresponding leg-equations $l_i := l_{P_i,Q_i}$ from Equation~\eqref{eqSpherical} 
and the linear space $L := V(l_1,\dotsc,l_n) \subset \PP^{16}$ where all leg-equations vanish. 
The pod is called \emph{mobile} if the intersection~$\SEbar \cap L$ contains a curve~$D$ that passes through the identity element of~$\SE_3$.
\end{definition}

\begin{remark}
 Notice that the notion of mobility we have just introduced does not require the curve~$D$ to be constituted of a one-dimensional set of real points; 
 one might say that this is a notion of ``complex mobility''.
\end{remark}

We prefer a slightly more intrinsic view of this situation.

\begin{definition}
A curve $D \subset \SEbar$ is called an $n$-pod curve, if
\begin{enumerate}
\item $D$ passes through the identity element of~$\SE_3$,
\item the space of linear equations in~$I_D$ contains $n$ leg-equations.
\end{enumerate}
\end{definition}

\begin{remark}
For each $n$-pod curve one can obtain several mobile $m$-pods with $m<n$ by choosing~$m$ of
the possible~$n$ leg-equations.
\end{remark}

Let us describe the set of leg-equations geometrically:

\begin{proposition}
We interpret a leg-equation~$l_{P,Q}$ as a point in the dual space~$\PPbondDual$. The set of 
all leg-equations is then contained in the $\PP^3 \times \PP^3 \stackrel{\mathrm{Segre}}{\hookrightarrow} \PP^{15} \subset \PPbondDual$
defined by the $2 \times 2$-minors of the matrix
\[
	\begin{pmatrix}
		A^* & y^* \\
		(x^*)^t & -2r^*
	\end{pmatrix}
\]
and the linear equation
\[
	a^*_{1,1}+a^*_{2,2}+a^*_{3,3} + h^* = 0
\]
where $a_{i,j}^*$, $x_i^*$, $y_i^*$, $r^*$, and~$h^*$ are the dual coordinates in~$\PPbondDual$.
\end{proposition}

\begin{proof}
By Proposition~\ref{pLegEquations} the coefficient matrix above is 
\[
	\begin{pmatrix}
		A^* & y^* \\
		(x^*)^t & -2r^*
	\end{pmatrix} = 
	\begin{pmatrix} Q \\ 1 \end{pmatrix} 
	\begin{pmatrix} P^t & 1 \end{pmatrix} .
\]
It has rank~$1$ and therefore the $2 \times 2$-minors of this matrix vanish.
Again by Proposition~\ref{pLegEquations}, the coefficient $h^*$ of~$h$ in a leg-equation
$l_{P,Q}$ is $-P_1Q_1-P_2Q_2-P_3Q_3$. Since $a^{*}_{i,i} = P_iQ_i$ for all~$i$, we get the linear equation.
\end{proof}
	
\begin{remark}
The linear equation $a^*_{1,1} + a^*_{2,2} + a^*_{3,3} + h^* = 0$ just says that the locus 
$\{ l_{P,Q} = 0 \}$ contains the identity element of~$\SE_3$.
\end{remark}	

Almost all points on the $\PP^3 \times \PP^3 \subset \PPbondDual$ come from leg-equations:

\begin{definition}
Consider the $\PP^3 \times \PP^3$ defined above. 
We call
\[
	\PP^3 \times \PP^3 \cap \{r^*=0\}
\]
the \emph{border} and
\[
	\PP^3 \times \PP^3 \cap \{r^*\not =0\}
\]
the \emph{interior} of~$\PP^3 \times \PP^3$.
\end{definition}

\begin{remark}
The leg-equations of Proposition~\ref{pLegEquations} are precisely those 
corresponding to points in the interior of~$\PP^3 \times \PP^3$.
\end{remark}

To be able to use the methods of algebraic geometry in the classification of hexapods, 
we slightly weaken our definition of $n$-pod curves:

\begin{definition}
Let $D \subset \SEbar \subset \PP^{16}$ be a curve on~$\SEbar$ defined over~$\CC$ 
and passing through the identity element of~$\SE_3$. 
We denote by 
\[
	\langle D \rangle \subset \PP^{16}
\]
the \emph{linear span} of~$D$, i.e., the smallest linear subspace of~$\PP^{16}$ containing~$D$. 
We denote by
\[
	\langle D \rangle^\perp := \PP H^0 \bigl( I_D(1) \bigr) \subset \PPbondDual
\]
the \emph{space of linear equations containing~$D$}.

We call $D$ a \emph{generalized $n$-pod curve} if
\[
	\langle D \rangle^\perp \cap (\PP^3 \times \PP^3) \subset \PPbondDual
\]
contains a scheme of length~$n$. 
\end{definition}

\begin{remark}
All $n$-pod curves are generalized $m$-pod curves for some $m \ge n$.
Conversely, a generalized $m$-pod curve $D$ is an ordinary $n$-pod curve if 
\[
	\langle D \rangle^\perp \cap (\PP^3 \times \PP^3)
\]
contains $n$ reduced real points in the interior of~$\PP^3 \times \PP^3$.
\end{remark}

The existence of generalized $n$-pod curves can often be guaranteed by a dimension count:

\begin{proposition} \label{pDimensionCount}
Consider an algebraic family of curves $\sD$ in~$\SEbar$, each containing the identity element of~$\SE_3$, 
and parametrized by an irreducible projective algebraic variety~$B$, i.e., a diagram
\xycenter{ 
	\id \ar[r] & D_b \ar[d] \ar[r]& \sD \ar[d] \ar[r] & \SEbar \ar@{ (->}[r] & \PP^{16} \\
	& b \ar[r] & B 	
}
where $D_b$ is the fiber over a point $b \in B$. Assume furthermore that
\[
	\codim \langle D_b \rangle = \gamma
\]
and that $\dim \bigl( \langle D_b \rangle \cap \SEbar \bigr) = 1$ for a general~$b \in B$.

\begin{enumerate}
\item \label{ic11} If $\gamma \ge 11$ then all $D_b$ are generalized $\infty$-pod curves. 
\item \label{ic10} If $\gamma = 10$ then all $D_b$ are either generalized $20$-pod curves, or $\infty$-pod curves. 
\item \label{ic6to9} If $6 \le \gamma \le 9$ and $\dim B \ge 6(10-\gamma)$ then the expected dimension
of the subfamily $B' \subset B$ of generalized hexapod curves is $\dim B - 6(10-\gamma)$. 
I.e., if there exists one generalized hexapod curve $D_b$ with $\codim \langle D_b \rangle = \gamma$, 
then there exists at least a $\bigl( \dim B - 6(10-\gamma) \bigr)$-dimensional family of generalized hexapod curves containing~$D_b$. 
\end{enumerate}
\end{proposition}

\begin{proof}
First notice that by semicontinuity we have
\[
	\codim \langle D_b \rangle \ge \gamma
\]
for all~$b$ in~$B$. We consider the open subvariety
\[
	B_\gamma = \{ b \in B \,|\, \codim \langle D_b \rangle = \gamma \} \subset B.
\]
Notice furthermore that
\[
	\dim \langle D_b \rangle^\perp = \codim \langle D_b \rangle - 1 = \gamma - 1.
\]
Since $D_b$ contains the identity element, we have $\langle D \rangle^\perp \subset \PP^{15} \subset \PPbondDual$
with~$\PP^{15}$ defined by the equation 
\[
	a^*_{1,1}+a^*_{2,2}+a^*_{3,3} + h^* = 0.
\]
Therefore
\[
	\codim \langle D_b \rangle^\perp = 15-\gamma+1 = 16-\gamma
\]
in this~$\PP^{15}$.

In case (\ref{ic11}), we have
\[
	 \codim \langle D_b \rangle^\perp = 16 - \gamma \le 5.
\]
Therefore $\langle D_b \rangle^\perp$ must intersect $\PP^3 \times \PP^3$ in at least a curve. 
Consequently, $D_b$ is a generalized $\infty$-pod curve.

In case (\ref{ic10}), we have $\codim \langle D_b \rangle^\perp = 6$ and
\[
	\dim \left( \langle D_b \rangle^\perp \cap \PP^3 \times \PP^3 \right) \ge 0
\]
If the intersection is finite, the length of the intersection is equal to the degree of $\PP^3 \times \PP^3 \subset \PP^{15}$, which is $20$. 
Otherwise, we again have generalized $\infty$-pod curves.

For case (\ref{ic6to9}), let $\GG := \GG(\gamma,16)$ be the Grassmannian parametrizing $\PP^{\gamma-1}$'s in~$\PP^{15}$. 
We then have a natural morphism
\begin{align*}
	\delta \colon B_\gamma &\to \GG \\
			      b &\mapsto \langle D_b \rangle^\perp
\end{align*}
and we define $B_\GG := \overline{\delta(B_\gamma)} \subset \GG$ as the closure of the image of~$B_\gamma$ in~$\GG$. 
The set~$B_\GG$ is an irreducible variety with $\dim B_\GG = \dim B$ because $\langle D_b \rangle \cap \SEbar$ is of dimension~$1$ for a general $b \in B$. 
Points in~$B_\GG$ correspond to~$\PP^{\gamma-1}$'s contained in some $\langle D_b\rangle^\perp$. 
For curves $D_b$ with $\dim \langle D_b \rangle^\perp > \gamma-1$ only those subspaces appear, that
are limits of subspaces occurring for curves~$D_{b'}$ with $\dim \langle D_{b'} \rangle^\perp = \gamma-1$.

We now look at those $\PP^{\gamma-1} \subset \PP^{15}$ that intersect the Segre embedding of~$\PP^3 \times \PP^3$ in at least $6$ points.
More precisely, we let
\[
	X_6 := \bigl\{ (l_1, \dotsc, l_6) \in (\PP^3 \times \PP^3)^6 \,\bigr|\, \langle l_1, \dotsc, l_6 \rangle = \PP^5 \bigr\}
\]
be the set of $6$-tuples of distinct points in~$\PP^3 \times \PP^3$ that span a~$\PP^5$. 
Since the latter condition is true for a general choice of such points, we have
\[
	\dim X_6 = 6(3+3) = 36
\]
and $X_6$ is irreducible. Now consider the following incidence variety
\[
	Y_6 = \{ (l_1, \dotsc, l_6, L) \,|\, \langle l_1, \dotsc, l_6 \rangle \subset L \} \subset X_6 \times \GG
\]
with its natural projections
 \xycenter{
	& Y_6 \ar[dl]_{\pi_1} \ar[dr]^{\pi_2}\\
	X_6 && \GG.
}

A fiber $\pi_1^{-1}(l_1,\dots,l_6)$ is the set of all $\PP^{\gamma-1}\subset \PP^{15}$ containing the $\PP^5$ spanned by the~$\{l_i\}_{i=1}^{6}$. 
This is a Schubert variety in~$\GG$. 
Projecting from the~$\PP^5$, hence landing on a~$\PP^9$, maps each~$\PP^{\gamma-1}$ containing the~$\PP^5$ 
to a $\PP^{\gamma-7} \subset \PP^9$, and the correspondence is~$1:1$. 
Therefore
\begin{align*}
	\dim \pi_1^{-1}(l_1,\dots,l_6) &= \dim \GG(\gamma-7,9) \\
	&= (\gamma-7+1)\bigl(10-(\gamma-7+1)\bigr) = (\gamma-6)(16-\gamma) \,.
\end{align*}
Since the fibers are of constant dimension and irreducible, the incidence variety~$Y_6$ is also irreducible with
\[
	\dim Y_6 = (\gamma-6)(16-\gamma) + 36 \,.
\]
We set $Y_\GG := \overline{\pi_2(Y_6)} \subset \GG$. Now $Y_6$ is irreducible and there exist linear spaces $\PP^{\gamma-1} \subset \PP^{15}$ that intersect $\PP^3 \times \PP^3$ in exactly $6$ distinct, linearly independent points.
Therefore $\pi_2$ is generically finite. It follows that
\[
	\dim Y_\GG = \dim Y_6 = (\gamma-6)(16-\gamma) + 36 \,.
\]
Now $\dim \GG = \gamma(16-\gamma)$ and we have
\begin{align*}
	\codim_\GG Y_\GG 
	&= \gamma(16-\gamma)-(\gamma-6)(16-\gamma) - 36 \\
	&= 6(16-\gamma) - 36 \\
	&= 6(10-\gamma) \,.
\end{align*}
We now consider the intersection $Z_\GG := B_\GG \cap Y_\GG \subset \GG$, 
and its decomposition into irreducible components
\[
	Z_\GG = Z_{\GG,1} \cup \dots \cup Z_{\GG,k}
\]
for some $k \in \NN$. 
By the dimension count above, the expected dimension of each~$Z_{\GG,i}$ is at least
\[
	\dim B_\GG - \codim Y_\GG = \dim B - 6(10-\gamma).
\]
Let $Z_i := \delta^{-1}(Z_{\GG,i}) \subset B$ be the preimage of~$Z_i$ in~$B$.
By construction, the elements $b \in Z_i$ parametrize generalized $n$-pod curves $D_b$
with $n\ge 6$. If a component~$Z_i$ contains at least one~$b$ such that 
$\dim \langle D_b \rangle = \gamma$, then $Z_i$ is non-empty, $\delta$ is generically~$1:1$ on~$Z_i$, and
\[
	\dim Z_i \ge \dim B - 6(10-\gamma). \qedhere
\]
\end{proof}

\section{The border of \texorpdfstring{$\SEbar$}{the closure of SE3}} \label{sInfinity}

In this section we describe the compactification of~$\SE_3$ more precisely,
with particular attention to the points in~$\SEbar$ that do not come from isometries in~$\SE_3$.

\begin{definition} \label{dSEinfty}
Denote by $\SEinfty := \SEbar \backslash \SE_3$ the \emph{border of~$\SEbar$}. Similarly,
denote by $\EulerInfty := \{e_0^2+e_1^2+e_2^2+e_3^2=0\}$ the \emph{border of Euler's~$\PP^3$}.
\end{definition}

\begin{definition} \label{dContact}
Let $\PP$ be a (possibly weighted) projective space, $X \subset \PP$ an irreducible variety, and $T \subset \PP$ a hypersurface. 
Let $Z := X \cap T$ be the set-theoretic intersection of~$T$ and~$X$ with the reduced scheme structure. 
The hypersurface~$T$ is called a \emph{contact hypersurface} of~$X$ if $T$ is tangent to~$X$ at each smooth point of~$Z$. 
In this case $Z$ is called a \emph{contact divisor}.
\end{definition} 

\begin{proposition} \label{pContactH}
Let $T \subset \PP^{16}$ and $T' \subset \weightedPP$ be the hypersurfaces defined by $h=0$ and $e_0^2+e_1^2+e_2^2+e_3^2=0$, respectively. 
Then $T$ is contact to $\SEbar \subset \PP^{16}$ and $T'$ is contact to $\SEbar \subset \weightedPP$. 
In both cases, we have that $\SEinfty$ is the contact divisor.
Furthermore, the ideal of $\SEinfty \subset \weightedPP$ is generated by
\begin{enumerate}
\item $e_0^2+e_1^2+e_2^2+e_3^2$,
\item the $4 \times 4$ Pfaffians of
\[
	\begin{pmatrix}
		0&
		x_{1}&
		x_{2}&
		x_{3}&
		e_{0}\\
		{-x_{1}}&
		0&
		{-x_{3}}&
		x_{2}&
		e_{1}\\
		{-x_{2}}&
		x_{2}&
		0&
		{-x_{1}}&
		e_{2}\\
		{-x_{3}}&
		{-x_{2}}&
		x_{1}&
		0&
		e_{3}\\
		{-e_{0}}&
		{-e_{1}}&
		{-e_{2}}&
		{-e_{3}}&
		0\\
	\end{pmatrix},
\]
\item the $4 \times 4$ Pfaffians of
\[
	\begin{pmatrix}
		0&
		y_{1}&
		y_{2}&
		y_{3}&
		e_{0}\\
		{-y_{1}}&
		0&
		y_{3}&
		{-y_{2}}&
		e_{1}\\
		{-y_{2}}&
		{-y_{3}}&
		0&
		y_{1}&
		e_{2}\\
		{-y_{3}}&
		y_{2}&
		{-y_{1}}&
		0&
		e_{3}\\
		{-e_{0}}&
		{-e_{1}}&
		{-e_{2}}&
		{-e_{3}}&
		0\\
	\end{pmatrix}.
\]
\end{enumerate}
Notice that the degree $4$ Pfaffians of the skew-symmetric matrices above are $x_1^2+x_2^2+x_3^2$ and $y_1^2+y_2^2+y_3^2$, respectively.
\end{proposition}

\begin{proof}
It follows from Definition~\ref{dBondEmbedding} that, set-theoretically, 
\[
	\SEbar \cap T = \SEinfty \,.
\]
Notice that $\SEbar \subset \weightedPP$ and $T' \subset \weightedPP$ are pullbacks of 
$\SEbar \subset \PP^{16}$ and $T \subset \PP^{16}$ via $\weightedEmbedding \colon \weightedPP \to \PP^{16}$ from Definition~\ref{dWeightedEmbedding}. 
Since $\weightedEmbedding$ is an immersion, it is enough to prove the statement for~$T'$. 

Let $I$ be the ideal described in the proposition. 
Firstly, $I$ is the ideal of a reduced variety if and only if
\[
	\sat \bigl( \rad(I) \bigr) = I \, ,
\]
where $\sat \bigl( \rad(I) \bigr)$ is the saturation of the radical ideal of~$I$. 
Secondly, $T'$ is contact to~$\SEbar$ in~$\SEinfty$ if and only if
\[
	\sat \bigl( I^2 + J+ (e_0^2+e_1^2+e_2^2+e_3^2) \bigr) = J + (e_0^2+e_1^2+e_2^2+e_3^2)
\]
where $J$ is the ideal of $\SEbar \subset \weightedPP$ described in Proposition
\ref{pEquationsWeightedPP}. Both conditions can be easily checked with a computer algebra system.
\end{proof}

\begin{corollary} \label{cContactC}
Let $D \subset \SEbar$ be a curve, and $P \in D \cap \SEinfty$ be a point at the border. 
Let $\pi \colon D \to \PP^3$ be the projection to Euler's~$\PP^3$.
If $\SEbar$ is smooth in~$P$ then $\pi(D)$ is either singular in~$\pi(P)$ or contact to~$\EulerInfty$ in~$\pi(P)$. 
In both cases the intersection multiplicity of~$\pi(D)$ with $\EulerInfty$ is at least~$2$.
\end{corollary}

\begin{proof}
Since $\SEbar$ is smooth in~$P$, by Proposition~\ref{pContactH} we have the following inclusions of tangent spaces
\[
	T_{D,P} \subset T_{\SEbar,P} \subset T_{T',P} \,.
\]
Notice that $\pi(T') = \EulerInfty$ is a smooth quadric. 
If in addition~$\pi(D)$ is smooth in~$\pi(P)$ we obtain
\[
	T_{\pi(D),\pi(P)} = \pi(T_{D,P}) \subset \pi(T_{T',P}) = T_{\pi(T'),\pi(P)} = T_{\EulerInfty,P} \,.
\]
namely $\pi(D)$ is contact to~$\EulerInfty$ at~$\pi(P)$, as claimed.
\end{proof}

\begin{remark}
The condition that $\SEbar$ is smooth in~$P$ is really necessary. 
Consider, in fact, the following example (see Figure~\ref{fContact}).

Let $X \subset \RR^3$ be the cone defined by $(x-z)^2+y^2 = z^2$. The variety~$X$ is singular at the origin. 
Let furthermore $T$ be the plane defined by $x=0$. Then $T$ is contact to~$X$ in the line~$E$ defined by $x=y=0$. 
Finally, let $\pi \colon \RR^3 \to \RR^2$ be the projection to the $(x,y)$-plane. The projection $\pi(T)$ is the coordinate line $x=0$.

We now look at two curves: the line $D_1$ defined by $x-2z=y=0$ and the conic~$D_2$ defined by $(x-1)^2+y^2=1$.
The curve~$D_1$ passes through the singular point of~$X$ while $D_2$ intersects~$T$ in a smooth point.
Projecting to the $(x,y)$-plane, we obtain
\[
	\pi(D_2) = \{ (x-1)^2+y^2=1 \}
\]
and
\[
	\pi(D_1) = \{ y=0 \} \,.
\]
The first is, as in Corollary~\ref{cContactC}, contact to~$\pi(T)$ at the origin, but the second intersects~$\pi(T)$ transversally.
\end{remark}

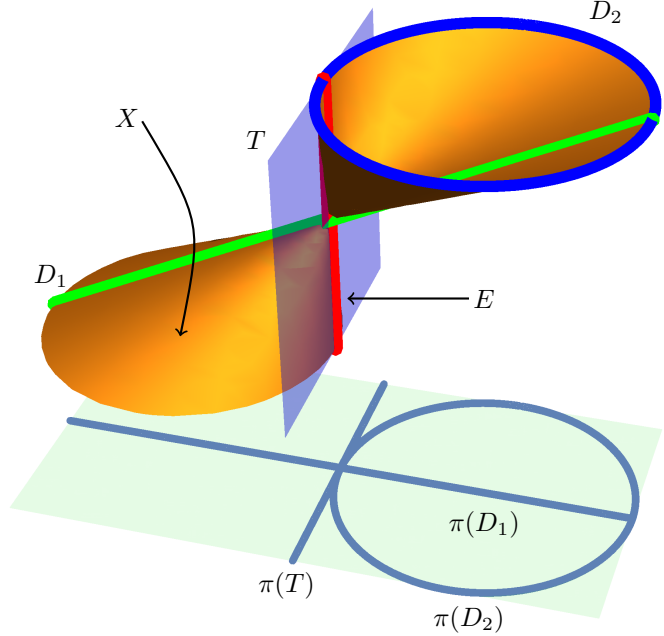
\begin{figure}[H]
\centering
\begin{tikzpicture}
 \node at (0,0) {\includegraphics[width=.8\textwidth]{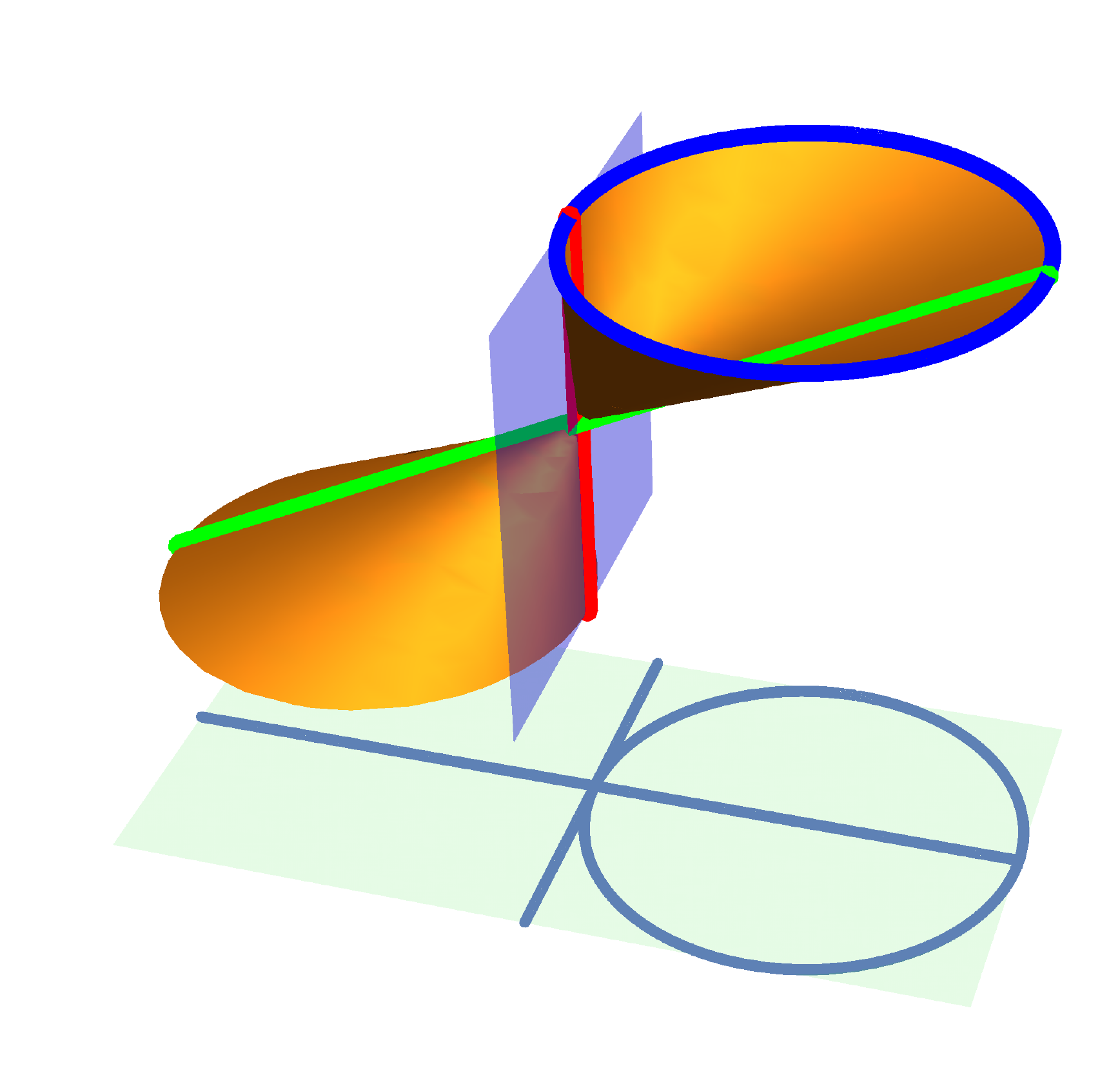}};
 \node at (-3.5, 0.3) {$D_1$};
 \node at (3.8, 3.8) {$D_2$};
 \node at (-0.8, 2.1) {$T$};
 \node at (-2.5, 2.4) {$X$};
 \node at (2.2, 0) {$E$};
 \node at (2.2, -3) {$\pi(D_1)$};
 \node at (2, -4.25) {$\pi(D_2)$};
 \node at (-0.4, -3.8) {$\pi(T)$};
 \draw[thick,->] (-2.3,2.35) .. controls (-1.5,1) .. (-1.8,-0.5); 
 \draw[thick,->] (2, 0) -- (0.4,0);
\end{tikzpicture}
\caption{Behavior of curves and contact surfaces under projections.}
\label{fContact}
\end{figure}
 
In \cite{BondTheory} the following stratification of the boundary $\SEinfty$ is introduced.

\begin{definition} \label{dBonds}
Let $P \in \SEinfty \subset \PP^{16}$ be a point on the border with coordinates $(A: x: y: r: 0)$ and consider the matrix
\[
	N := rA + 2yx^t \,.
\]
Then $P$ is called
\begin{itemize}
\item an \emph{inversion bond} if $A \not= 0$ and $N \not= 0$,
\item a \emph{butterfly bond} if $A \not= 0$ and $N=0$,
\item a \emph{similarity bond} if $A = 0$, $x \not= 0$ and $y \not= 0$,
\item a \emph{collinearity bond} if $A=0$ and either $x=0$ and $y\not=0$, or $y=0$ and $x\not=0$.
\item the \emph{vertex bond} if $A=x=y=0$.
\end{itemize}
One can check that the singular locus of~$\SEbar$ consists precisely of \emph{butterfly bonds}, \emph{collinearity bonds}, and the \emph{vertex}. See Figure~\ref{fStratification} for a visualization of this stratification.
\end{definition}

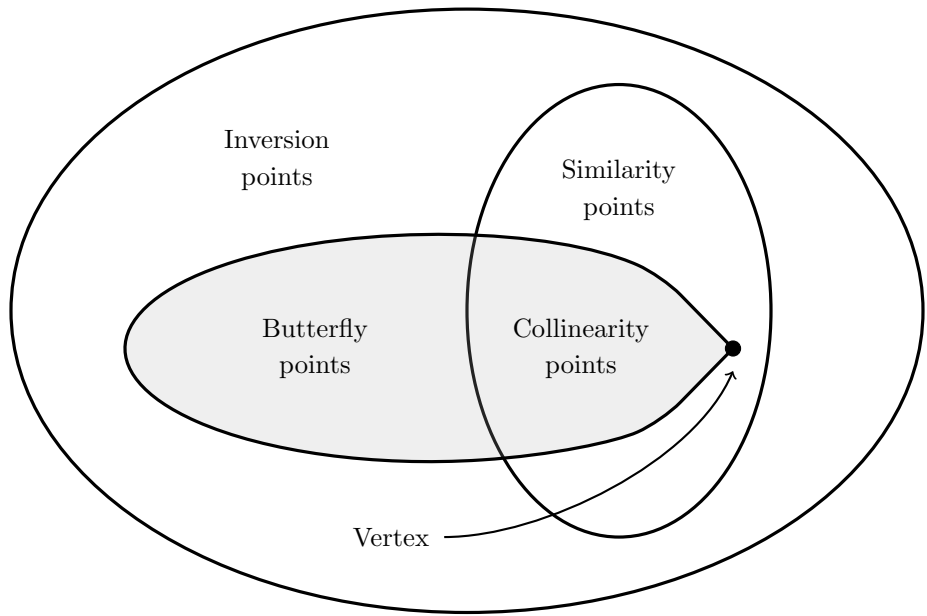
\begin{figure}
\centering
\begin{tikzpicture}
 \draw[very thick] (0,0) ellipse (6cm and 4cm);
 \draw[very thick] (2,0) ellipse (2cm and 3cm);
 \draw[very thick, fill=black!30!white, fill opacity=0.2, rounded corners=10pt] (-0.5, -0.5) +(40:4cm and 1.5cm) arc (40:320:4cm and 1.5cm) [sharp corners] -- (3.5,-0.5) [rounded corners=10pt] -- cycle;
 \draw[fill=black] (3.5,-0.5) circle (.1cm);
 \node[align=center] at (-2,-0.5) {Butterfly \\ points};
 \node[align=center] at (1.5,-0.5) {Collinearity \\ points};
 \node[align=center] at (-2.5,2) {Inversion \\ points};
 \node[align=center] at (2,1.6) {Similarity \\ points};
 \node[align=center] at (-1,-3) {Vertex};
 \draw[thick,->] (-0.3,-3) .. controls (1,-3) and (3, -2) .. (3.5,-.8);
\end{tikzpicture}
\caption{The stratification of the boundary $\SEinfty$. The gray area is the singular locus of~$\SE_3$. The ``wedge'' at the vertex denotes that the singular locus is in turn singular at the vertex.}
\label{fStratification}
\end{figure}
 
\begin{definition}
\label{dBondSpan}
Let $D \subset \SEbar \subset \PP^{16}$. 
We say that $D$ has an inversion, butterfly, similarity, collinearity, or vertex bond~$P$
if and only if
\[
	P \in \langle D \rangle.
\]
\end{definition}

\begin{remark}
In \cite{BondTheory}, bonds are related to hexapods, and not configuration curves.
This means that in \cite{BondTheory} a point~$P \in \SEbar$ is a bond for a hexapod with leg-equations $\{ l_i \}_{i=1}^6$ if
\[
	P \in V(l_1,\dots,l_6) .
\]
This condition is less strong than the one we ask in this paper. In fact, if $D$ is a configuration curve of the same hexapod 
and $\codim \langle D \rangle \ge 7$ then the inclusion
\[
	\langle D \rangle \subset V(l_1,\dots,l_6)
\]
is strict, and so points that are bonds for the hexapod according to \cite{BondTheory} 
may not be bonds for~$D$ according to the definition that we use in this paper.
\end{remark}

\begin{remark}
In \cite{BondTheory} the authors prove strong geometric restrictions for hexapods having butterfly or collinearity bonds.
\end{remark}

We can now reformulate Corollary~\ref{cContactC}:

\begin{corollary} \label{cContactButterfly}
Let $D \subset \SEbar$ be a curve without butterfly, collinearity or vertex bonds.
Let $\pi \colon D \to \PP^3$ be the projection to Euler's~$\PP^3$.
Then $\pi(D)$ is contact to~$\EulerInfty$.
\end{corollary}

\section{Pfaffian equations} \label{sPfaffianEquations}

As we have seen in Proposition~\ref{pEquationsWeightedPP} and Proposition~\ref{pContactH}, 
Pfaffian equations feature prominently in the description of~$\SEbar$ and its border~$\SEinfty$. 
In this section we establish some geometric consequences from the existence of Pfaffian equations in the ideal of a space curve. 
We will use these consequences to prove our Classification Theorem~\ref{tClassification}.

\begin{notation}
In this paper we often deal with Betti tables and minimal graded resolutions of modules and sheaves. 
In these cases, we adopt the so-called ``Macaulay notation'', 
which imposes to write resolutions with arrows pointing to the left; 
in this way, in fact, it is easy to read off the Betti table from the resolution. 
When exact sequences are not related to resolutions, 
we use the standard convention of writing them with arrows pointing to the right.
\end{notation}

Let $K$ be an algebraically closed field and let $K[e_0,\dots,e_3]$ be the coordinate ring of~$\PP^3$.
In this section we consider skew-symmetric matrices
\[
	M := \begin{pmatrix}
	           0 & z_{0,1} & z_{0,2} & z_{0,3} & e_0 \\
	 -z_{0,1} &          0 & z_{1,2} & z_{1,3} & e_1 \\
	 -z_{0,2} & -z_{1,2} &        0  & z_{2,3} & e_2 \\
	 -z_{0,3} & -z_{1,3} & -z_{2,3} & 0 & e_3 \\
	 -e_0 & -e_1 & -e_2 & -e_3 & 0
	 \end{pmatrix}
\]
with $z_{i,j}$ homogeneous polynomials of degree~$2$.

Let $P_i := P_i(M)$ be the Pfaffian of the $4 \times 4$ matrix~$M_i$ 
obtained by removing the $i$-th row and column of~$M$. 
Then $P_0, \dotsc, P_3$ are of degree~$3$ and $P_4$ is of degree~$4$.

Let furthermore $I_M := (P_0,\dots,P_4)$ be the ideal generated by the $4 \times 4$ Pfaffians of~$M$, 
and $X_M := V(I_M) \subset \PP^3$ be the scheme defined by this ideal. 

The following is well-known:

\begin{proposition} \label{pComplex}
With the notations above we have a complex
\[
	\sO_{X_M} \leftarrow \sO_{\PP^3} 
	\xleftarrow{\mathbf{P}} \begin{matrix} 4\sO_{\PP^3}(-3) \\ \oplus \\ \sO_{\PP^3}(-4) \end{matrix} 
	\xleftarrow{M} \begin{matrix} \sO_{\PP^3}(-4) \\ \oplus \\ 4\sO_{\PP^3}(-5) \end{matrix} 
	\xleftarrow{\mathbf{P}^t} \sO_{\PP^3}(-8) 
	\leftarrow 0
\]
with $\mathbf{P} = (P_0,-P_1,P_2,-P_3,P_4)$.

If $X_M$ is zero-dimensional, this complex is exact and a minimal free resolution of~$\sO_{X_M}$.
\end{proposition}

\begin{proof}
See \cite[Section~3]{BuEi}.
\end{proof}

\begin{remark}
In this section we will be interested in the case where $X_M$ is NOT zero-dimensional.
\end{remark}

\begin{proposition} 
\label{pDeg3impliesDeg4}
Let $Z \subset \PP^3$ be an integral scheme such that the four degree~$3$ Pfaffians of~$M$ vanish on~$Z$. 
Then also the degree~$4$ Pfaffian of~$M$ vanishes on~$Z$.
\end{proposition}

\begin{proof}
Let $a$ be a point of~$Z$. 
Then at least one coordinate of~$a$ is nonzero, 
so after possibly renumbering the variables we can assume $e_0(a) \not=0$. 
Now $\mathbf{P}^t \cdot M = 0$ by Proposition~\ref{pComplex}. 
Considering the first column of~$M$, we obtain
\[ 
	P_1 z_{0,1} - P_2 z_{0,2} + P_3 z_{0,3} - P_4e_0 = 0.
\]
Evaluating this at~$a$ gives
\[
	P_4(a) e_0(a) = 0.
\]
Since $e_0(a) \not= 0$, it follows that $P_4(a) = 0$.
Therefore $P_4$ vanishes on~$Z$.
\end{proof}

\begin{proposition} \label{p3pfaffians}
Assume that the space of the degree $3$ polynomials in~$I_M$ is at most $3$-dimensional. 
Then there exists a $2 \times 3$ matrix
\[
	N = 
	\begin{pmatrix}
		l_0 & l_1 & l_2 \\
		q_0 & q_1 & q_2 \\
	\end{pmatrix}
\]
with $l_1, l_2,l_3$ of degree~$1$ and linearly independent, and $\deg q_i = 2$ such that 
the ideal~$I_M$ generated by the $4 \times 4$ Pfaffians of~$M$ is the same as 
the one generated by the $2 \times 2$ minors of~$N$.

If $X_M$ is a curve, it has degree~$7$, arithmetic genus~$5$ and Betti table
\[
	\begin{matrix} 1 & - & -   \\  - & - & -  \\ - & 3 & 1 \\ - & - & 1    \end{matrix} \ .
\]
\end{proposition}

\begin{proof}
This is a straightforward computation.

All degree~$3$ polynomials in~$I_M$ are generalized Pfaffians, i.e., 
Pfaffians of some $M'$ obtained form $M$ by simultaneous row and column operations. 
Since there is a $4$-dimensional space of generalized degree~$3$ Pfaffians, 
there must be at least one generalized Pfaffian that vanishes identically. 
Now the claim above is invariant under automorphisms of~$\PP^3$ and simultaneous row and column operations on~$M$, 
so we can assume that $P_0 = 0$, namely
\[
	{e}_{3} {z}_{1,2} - {e}_{2} {z}_{1,3} + {e}_{1} {z}_{2,3} = 0 \,.
\]
We can write this equation as
\[
	(e_3,-e_2,e_1) \begin{pmatrix} z_{1,2} \\ z_{1,3} \\ z_{2,3} \end{pmatrix} = 0,
\]
i.e., $\left( \begin{smallmatrix} z_{1,2} \\ z_{1,3} \\ z_{2,3} \end{smallmatrix} \right)$ is a degree~$2$ syzygy of~$(e_3,-e_2,e_1)$.
The space of syzygies of~$(e_3,-e_2,e_1)$ is generated by the columns of 
\[
	\begin{pmatrix}
		0&
		{e}_{1}&
		{e}_{2}\\
		{-{e}_{1}}&
		0&
		{e}_{3}\\
		{-{e}_{2}}&
		{-{e}_{3}}&
		0\\
	\end{pmatrix} \,.
\]
It follows that there are linear polynomials $m_1,m_2,m_3$ such that
\[
	\begin{pmatrix} 
		z_{1,2} \\ z_{1,3} \\ z_{2,3} 
	\end{pmatrix} = 
	\begin{pmatrix}0&
		{e}_{1}&
		{e}_{2}\\
		{-{e}_{1}}&
		0&
		{e}_{3}\\
		{-{e}_{2}}&
		{-{e}_{3}}&
0
	\end{pmatrix}
	\begin{pmatrix} 
		m_3 \\ -m_2 \\ m_1 
	\end{pmatrix} =
	\begin{pmatrix}
		-{e}_{1} {m}_{2}+{e}_{2} {m}_{1}\\
		-{e}_{1} {m}_{3}+{e}_{3} {m}_{1}\\
		-{e}_{2} {m}_{3}+{e}_{3} {m}_{2}
	\end{pmatrix}.
\]
Substituting this into~$M$ we obtain
\[
	M' = 
	\begin{pmatrix}
		0&
		{z}_{0,1}&
		{z}_{0,2}&
		{z}_{0,3}&
		{e}_{0}\\
		{-{z}_{0,1}}&
		0&
		-{e}_{1} {m}_{2}+{e}_{2} {m}_{1}&
		-{e}_{1} {m}_{3}+{e}_{3} {m}_{1}&
		{e}_{1}\\
		{-{z}_{0,2}}&
		{e}_{1} {m}_{2}-{e}_{2} {m}_{1}&
		0&
		-{e}_{2} {m}_{3}+{e}_{3} {m}_{2}&
		{e}_{2}\\
		{-{z}_{0,3}}&
		{e}_{1} {m}_{3}-{e}_{3} {m}_{1}&
		{e}_{2} {m}_{3}-{e}_{3} {m}_{2}&
		0&
		{e}_{3}\\
		{-{e}_{0}}&
		{-{e}_{1}}&
		{-{e}_{2}}&
		{-{e}_{3}}&
        0
	\end{pmatrix}.
\]
We now consider the invertible matrix
\[
	B = 
	\begin{pmatrix}
		1&
		0&
		0&
		0&
		0\\
		0&
		1&
		0&
		0&
		{m}_{1}\\
		0&
		0&
		1&
		0&
		{m}_{2}\\
		0&
		0&
		0&
		1&
		{m}_{3}\\
		0&
		0&
		0&
		0&
1
	\end{pmatrix}
\]
and, recalling that for a $4 \times 4$ skew-symmetric matrix~$A$ it holds 
$\mathrm{Pf}(CAC^t) = \det(C) \mathrm{Pf(A)}$, we compute the matrix $M'' = B M' B^t$:
\[
	\begin{pmatrix}
		0&
		{z}_{0,1}+{e}_{0} {m}_{1}&
		{z}_{0,2}+{e}_{0} {m}_{2}&
		{z}_{0,3}+{e}_{0} {m}_{3}&
		{e}_{0}\\
		-{z}_{0,1}-{e}_{0} {m}_{1}&
		0&
		0&
		0&
		{e}_{1}\\
		-{z}_{0,2}-{e}_{0} {m}_{2}&
		0&
		0&
		0&
		{e}_{2}\\
		-{z}_{0,3}-{e}_{0} {m}_{3}&
		0&
		0&
		0&
		{e}_{3}\\
		{-{e}_{0}}&
		{-{e}_{1}}&
		{-{e}_{2}}&
		{-{e}_{3}}&
0
	\end{pmatrix}.
\]
Reordering the rows and columns finally gives
\[
	M''' = 
	\begin{pmatrix}
		0&
		{-{e}_{0}}&
		{-{e}_{1}}&
		{-{e}_{2}}&
		{-{e}_{3}}\\
		{e}_{0}&
		0&
		{z}_{0,1}+{e}_{0} {m}_{1}&
		{z}_{0,2}+{e}_{0} {m}_{2}&
		{z}_{0,3}+{e}_{0} {m}_{3}\\
		{e}_{1}&
		-{z}_{0,1}-{e}_{0} {m}_{1}&
		0&
		0&
		0\\
		{e}_{2}&
		-{z}_{0,2}-{e}_{0} {m}_{2}&
		0&
		0&
		0\\
		{e}_{3}&
		-{z}_{0,3}-{e}_{0} {m}_{3}&
		0&
		0&
0
	\end{pmatrix}.
\]
The Pfaffians of~$M'''$ are now the same as the $2 \times 2$ minors of 
\[
	N = 
	\begin{pmatrix}
		{-{e}_{1}}&
		{-{e}_{2}}&
		{-{e}_{3}}
		\\
		{z}_{0,1}+{e}_{0} {m}_{1}&
		{z}_{0,2}+{e}_{0} {m}_{2}&
{z}_{0,3}+{e}_{0} {m}_{3}
	\end{pmatrix}.
\]
Since the Pfaffians of~$M'''$ generate the same ideal as the Pfaffians of~$M$, this proves the first claim. 

Associated to the matrix~$N$ is an Eagon-Northcott complex
\[
	0 \leftarrow I_X \leftarrow 3\sOPP(-3) \xleftarrow{N^t} \sOPP(-4) \oplus \sOPP(-5) \leftarrow 0 \,.
\]
If $X$ is of codimension~$2$, i.e., a curve, then the Eagon-Northcott complex is 
exact (see \cite[Appendix~2H, Theorem A2.60]{Eisenbud2005}; in this case this 
is the so-called Hilbert-Burch theorem). This gives the claimed Betti table. 
Degree and arithmetic genus can be read from the Betti table.
\end{proof}

\begin{proposition} \label{p2pfaffians}
Assume that the space of the degree~$3$ polynomials in~$I_M$ is at most $2$-dimensional. 
Then there exists linear polynomials~$l_0,l_1$ and a quadratic polynomial~$q$ such that
\[
	I_M = (l_0q,l_1q) \,.
\]
If the space of degree~$3$ polynomials in~$I_M$ is exactly $2$-dimensional then 
$X_M$ is the union of a quadric and a line with Betti table
\[
	\begin{matrix} 1 & - & -   \\  - & - & -  \\ - & 2 & 1   \end{matrix} \,.
\]
If the dimension of the space of degree~$3$ polynomials in~$I_M$ is less than~$2$, then $I_M = 0$.
\end{proposition}

\begin{proof}
By Proposition~\ref{p3pfaffians}, we can assume that~$I_M$ is generated by 
the $2 \times 2$ minors of a matrix~$N$
\[
	N = 
	\begin{pmatrix}
		l_0 & l_1 & l_2 \\
		q_0 & q_1 & q_2 \\
	\end{pmatrix}
\]
with $\deg l_i = 1$, the forms $l_0,l_1,l_2$ are linearly independent, and $\deg q_i = 2$. 
By our assumptions, at least one generalized $2 \times 2$ minor of~$N$ vanishes identically. 
After suitable column operations, we can assume this minor to be the determinant of the first two columns:
\[
	l_0 q_1 - l_1 q_2 = 0 \, .
\]
This implies that there is a linear polynomial~$m$ such that $q_1 = m \, l_1$ and $q_2 = m \, l_2$. 
Substituting this into~$N$ gives
\[
	N' = 
	\begin{pmatrix}
		l_0 & l_1 & l_2 \\
		ml_0 & ml_1 & q_2 \\
	\end{pmatrix} \,.
\]
After a suitable row operation we obtain
\[
	N'' = 
	\begin{pmatrix}
		l_0 & l_1 & l_2 \\
		0 & 0 & q_2-ml_2 \\
	\end{pmatrix} \,.
\]
Since the $2 \times 2$ minors of~$N$ are the same as those of~$N''$, 
this proves the first claim with $q = q_2 - m l_2$.
By Proposition~\ref{p3pfaffians}, the $l_i$ are linearly independent, 
so if $q \not\equiv 0$, we get the claimed resolution and $X_M$ is the union of~$\{ q=0 \}$ with $\{ l_0=l_1=0 \}$. 

Assume now that $l_0q$ and $l_1q$ are not independent. Since $l_0$ and $l_1$ are linearly independent, this implies $q\equiv0$ and therefore $I_M = 0$.
\end{proof}

\begin{table}
\caption{Properties of possible curves $C$ as in Proposition~\ref{pClassificationAlgebraic}. In the cases marked by ($\ast$), the scheme $X$ must be a complete intersection of a quadric and a quartic.}
\begin{tabular}{ccc|c|cc}
\toprule
 $c$ & $d$ & $g$& Betti table of~$C$ & name \\ 
 \midrule
 $ \ge 4 $ & & &  & a plane curve \\
 \midrule
 $3$ & 3 & 0 & $\begin{matrix} 1 & - & - \\  - & 3 & 2  \end{matrix}$ & rational normal curve\\
 \midrule
 $2$ & 4 & 1 & $\begin{matrix} 1 & - & - \\  - & 2& - \\ - & - & 1  \end{matrix}$ & elliptic normal curve\\
 \midrule
 $1$ & 4 & 0 & $\begin{matrix} 1 & - & -  & - \\  - & 1 & - & - \\ - & 3 & 4 & 1  \end{matrix}$ &  $(1,3)$ on $\PP^1 \times \PP^1$ \\
 \midrule
     & 5 & 2 & $\begin{matrix} 1 & - & -  \\  - & 1 & -  \\ - & 2 & 2  \end{matrix}$ & $(2,3)$ on $\PP^1\times \PP^1$ \\
 \midrule
     & 6 & 4 & $\begin{matrix} 1 & - & -  \\  - & 1 & -  \\ - & 1 & - \\ - & - & 1  \end{matrix} $ & canonical curve\\
 \midrule
     & 5 & 0 & $\begin{matrix} 1 & - & -  & - \\  - & 1 & - & - \\ - & - & - & - \\ - & 4 & 6 & 2  \end{matrix} \quad \hfill (\ast)$ & $(1,4)$ on $\PP^1\times \PP^1$\\
\midrule
     & 6 & 3 & $\begin{matrix} 1 & - & -  & - \\  - & 1 & - & - \\ - & - & - & - \\ - & 3 & 4 & 1 \end{matrix} \hfill (\ast)$  &  $(2,4)$ on $\PP^1\times \PP^1$ \\
 \midrule
     & 7 & 6 & $\hfill \begin{matrix} 1 & - & -  \\  - & 1 & -  \\ - & - & - \\ - & 2 & 2  \end{matrix} \hfill (\ast)$ &  $(3,4)$ on $\PP^1\times \PP^1$ \\
 \midrule
     & 8 & 9 & $\hfill \begin{matrix} 1 & - & -  \\  - & 1 & -  \\ - & - & -  \\ - & 1 & - \\ - & - & 1\end{matrix} \hfill (\ast)$ & $(4,4)$ on $\PP^1\times \PP^1$ \\
\bottomrule
\end{tabular}
\label{tWithQuadrics}
\end{table} 

\begin{table}
\caption{Properties of possible curves $C$ as in Proposition~\ref{pClassificationAlgebraic}.}
\begin{tabular}{ccc|c|cc}
\toprule
$c$ & $d$ & $g$ & Betti table of~$C$ & name \\ 
 \midrule
$0$ & 5 & 0 & $\begin{matrix} 1 & - & -  & - \\  - & - & - & - \\ - & 4 & 3 & - \\ - & 1 & 2 & 1  \end{matrix}$ & rational quintic\\
 \midrule
    & 5 & 1 & $\begin{matrix} 1 & - & -  & - \\  - & - & - & - \\ - & 5 & 5 & 1  \end{matrix}$ & elliptic quintic\\
 \midrule
    & 6 & 3 & $\begin{matrix} 1 & - & -   \\  - & - & -  \\ - & 4 & 3   \end{matrix}$ & determinantal sexic \\
 \midrule
    & 6 & 2 & $\begin{matrix} 1 & - & -  & - \\  - & - & - & - \\ - & 3 & 1 & - \\ - & 1 & 3 & 1  \end{matrix}$ &  \\
 \midrule
    & 7 & 5 & $\begin{matrix} 1 & - & -   \\  - & - & -  \\ - & 3 & 1 \\ - & - & 1    \end{matrix}$ &  determinantal septic\\
 \bottomrule
 \end{tabular}
 \label{tWithoutQuadrics}
\end{table}

\begin{proposition} \label{pClassificationAlgebraic}
Let $Q_1, \dotsc, Q_c$ be linearly independent quadratic polynomials in~$K[e_0,\dots,e_3]$, 
let $M$ be a skew-symmetric matrix as above, let $P_0,\dots,P_4$ be the $4 \times 4$ Pfaffians of~$M$, 
and $F$ be a polynomial of degree~$4$. 
Let $X \subset \PP^3$ be the scheme defined by the ideal $I_X := (Q_1, \dotsc, Q_c, P_0, \dotsc, P_4, F)$ and 
$C \subset \PP^3$ be an irreducible curve of degree~$d$ and arithmetic genus~$g$.

Assume that
\begin{enumerate}[label=(\arabic*)]
\item $C$ is an irreducible component of~$X$,
\item $H^0 \bigl( I_C(2) \bigr) = \langle Q_1, \dotsc, Q_c \rangle$.
\end{enumerate}
Then the possible invariants of~$C$ are listed in Table~\ref{tWithQuadrics} and Table~\ref{tWithoutQuadrics}.
\end{proposition}

\begin{proof}
If $c \ge 1$, the curve~$C$ lies on an irreducible quadric hypersurface~$Q \subset \PP^3$. 
Curves on irreducible quadrics are well-known, i.e., 
they are curves of bidegree~$(a,b)$ in~$\PP^1 \times \PP^1$ if $Q \subset \PP^3$ is smooth, or 
if $Q$ is singular they are curves with the same invariants and minimal free resolutions as those of bidegree~$(a,a)$ or $(a,a+1)$ on~$\PP^1 \times \PP^1$. 
By abuse of notation, we will also say that these latter curves are of bidegree~$(a,a)$ or~$(a,a+1)$.

Since $C$ is a component of~$X$, there is at least one other generator of~$I_X$ that does not vanish on~$Q$. 
The maximal degree of generators of~$I_X$ is~$4$ so $a,b \le 4$. This gives Table~\ref{tWithQuadrics}. 

If $C$ is of type $(a,4)$, there is no other quadric than~$Q$ in the ideal~$I_X$ and all cubic Pfaffians vanish on~$Q$. 
By Proposition~\ref{pDeg3impliesDeg4} this implies that the degree~$4$ Pfaffian also vanishes on~$Q$. 
Therefore $I_X = (Q, F)$ with $F$ a nonzero polynomial. 
Since $Q$ is irreducible, this shows that $X$ is a curve of bidegree~$(4,4)$ on~$Q$.

If $c=0$ the space of degree~$3$ polynomials in~$I_C$ contains the degree~$3$ Pfaffians of~$M$. 
Now $C$ is irreducible and does not lie on a quadric hypersurface. 
By Proposition~\ref{p2pfaffians} this implies that
$h^0\bigl(I_C(3)\bigr) \ge 3$. 

Naito \cite{deg6resolutions} has classified the possible minimal free resolutions of smooth, nondegenerate space curves of degree at most~$6$. 
His theorem is also true for irreducible nondegenerate space curves (see Appendix~\ref{sBetti} for the necessary changes in his proof). 
Looking at Naito's list, we see that in degree at most~$6$ there are only four possibilities with the above restrictions. 
We have listed them as the first $4$ entries of Table~\ref{tWithoutQuadrics}. 

If $d \ge 7$ and $h^0\bigl(I_C(3)\bigr) \ge 3$ then by Proposition~\ref{pSeptic} of Appendix~\ref{sThreeCubics} we can only have $d=7$ and $h^0\bigl(I_C(3)\bigr) = 3$.
In particular, at least one Pfaffian of~$M$ vanishes identically and we are in the situation of Proposition~\ref{p3pfaffians}.
The degree~$3$ Pfaffians of~$M$ are cubics in the ideal of~$I_C$. 
Since they span a $3$-dimensional space, the degree~$3$ Pfaffians of~$M$ are exactly the cubics generating~$I_C$. 
Therefore they cut out a curve (i.e.~$C$) and, again by Proposition~\ref{p3pfaffians}, 
the curve $C$ has invariants as in the last row of Table~\ref{tWithoutQuadrics}.
\end{proof}

\section{The Classification Theorem} \label{sClassification}

Here we classify the possible space curves that appear as~$\pi(D)$ for a hexapod curve~$D \subset \SEbar$ under suitable genericity conditions.
Recall that $\pi \colon \SEbar \dashrightarrow \PP^3$ is the projection on the Euler coordinates $e_0, \dotsc, e_3$ parametrizing
the rotational part of an isometry.

\begin{remark}
\label{rGeneric}
 If an irreducible curve $D \subset \SEbar \subset \PP^{16}$ is such that 
 its span does not intersect the center of the projection $\hat{\pi} \colon \PP^{16} \dashrightarrow \PP^9$ 
 --- equivalently, if the projection on the $(A,h)$-coordinates is an isomorphism on~$\langle D \rangle$ --- 
 then the dimension of~$\langle D \rangle$ is at most~$9$, 
 making $\codim \langle D \rangle$ in~$\PP^{16}$ at least~$7$.
\end{remark}

\begin{theorem} \label{tClassification}
Let $D \subset \SEbar \subset \PP^{16}$ be an irreducible curve such that 
the span~$\langle D \rangle$ does not intersect the center of the projection $\hat{\pi} \colon \PP^{16} \dashrightarrow \PP^9$. 
Assume that $\codim \langle D \rangle = c+7$ for some nonnegative number~$c$. 
Then $\pi(D)$ is one of the curves listed in Table~\ref{tWithQuadrics} and Table~\ref{tWithoutQuadrics}.
\end{theorem}

\begin{proof}
Consider
\[
	Y := \SEbar \cap \langle D \rangle \subset \PP^{16} \, .
\]
Since, by hypothesis, $\left\langle D \right\rangle$ is cut out in~$\PP^{16}$ by $c+7$ linear forms, the variety~$Y$ is cut out in~$\weightedPP$ by $c+7$ quadrics and the equations from Proposition~\ref{pEquationsWeightedPP}. 
Since $\left\langle D \right\rangle$ does not intersect the space~$\{ A=h=0 \}$, 
the pullback of the span $\langle D \rangle$ in~$\weightedPP$ does not intersect the locus $\{ e_0 = \dots = e_3 = 0\}$.
Moreover, after possibly a change of basis, the linear forms defining $\left\langle D \right\rangle$
can be taken to be $L_1, \dotsc, L_c$, $L_{c+1} - x_1, L_{c+2} - x_2, L_{c+3} - x_3$, $L_{c+4} - y_1, L_{c+5} - y_2, L_{c+6} - y_3$, $L_{c+7} - r$, where the $L_i$ are linear polynomials in the entries of~$A$ and in~$h$.
Hence, after choosing an appropriate basis in~$\weightedPP$, the space of quadrics cutting out~$Y$ is of the form
\begin{align*}
	\bigl\langle
		& Q_1,
		\dotsc,
		Q_c,\\
		& Q_{c+1} - x_1,
		Q_{c+2} - x_2,
		Q_{c+3} - x_3,\\
		& Q_{c+4} - y_1,
		Q_{c+5} - y_2,
		Q_{c+6} - y_3,\\
		&Q_{c+7} - r
	\bigr\rangle
\end{align*}
with $Q_1,\dots,Q_{c+7}$ degree~$2$ polynomials in~$e_0,\dotsc,e_3$.

Projecting via~$\pi$ then amounts to eliminating $\{ x_i,y_j \}_{i,j}$ and~$r$ in the equations from Proposition~\ref{pEquationsWeightedPP} using the relations above. 
For $\pi(D) \subset X :=\pi(Y)$ we are then in the situation of Proposition~\ref{pClassificationAlgebraic}. 
This proves the theorem.
\end{proof}

\section{Construction} \label{sConstruction}

In this section we construct, for each curve~$C \subset \PP^3$ in Table~\ref{tWithQuadrics} and Table~\ref{tWithoutQuadrics}, 
families of curves $D \subset \SEbar$ satisfying the conditions of Theorem~\ref{tClassification}; see Theorem~\ref{tFamilies}. 
We also calculate the dimension of these families; see Table~\ref{tWithQuadricsDimensions} and Table~\ref{tWithoutQuadricsDimensions}. 
Here we work over an algebraically closed field~$K$ (e.g.\ $\CC$) unless stated otherwise.

To get a handle on this problem, we compare three types of objects for each curve $C \subset \PP^3$:
\begin{enumerate}
\item Curves $D \subset \SEbar$ such that $\pi(D) = C$. These curves are called \emph{lifts} of~$C$.
\item $1 \times 6$ matrices of quadrics in the coordinate ring $R_C$ of~$C$, subject to certain algebraic conditions.
\item Sections $\tau$ of the tangent bundle $\TPthree$ of~$\PP^3$ restricted to~$C$, subject to certain geometric conditions.
\end{enumerate}

Theorem~\ref{tConstruction} shows that, under suitable assumptions on~$C$, there is a $1:1$ correspondence between the previous three types of objects.

We start by making the concept of a lift more precise:

\begin{definition} \label{dLift}
Let $C \subset \PP^3$ be a space curve. 
A curve $D \subset \SEbar$ is called a \emph{lift of~$C$ to~$\SEbar$} 
if and only if there exists a morphism $\sigma \colon C \to \SEbar$ 
such that $\sigma(C) = D$ and the following diagram commutes:
\xycenter{
	& \SEbar \ar[d]^{\pi}\\
	C \ar[ur]^\sigma \ar@{ (->}[r] & \PP^3
}
We denote by 
\[
	\sigma^\# : R_{\weightedPP} \to R_C
\]
the associated homomorphism of graded rings. 
\end{definition}

\begin{remark} \label{rSigmaHash}
Let $C\subset \PP^3$ be a space curve and let
\[
	\sigma \colon C \to \SEbar \subset \weightedPP
\]
be a morphism defining a lift of~$C$ to~$\SEbar$.
Since $\pi \circ \sigma = \id_C$, we have $\sigma^{\#}(e_i) = e_i$ for $i=0,\dots,3$. 
From the homogeneity of~$\sigma^\#$ it follows that
\[
	\sigma^\#(x_1,x_2,x_3,y_1,y_2,y_3)
\]
is a $1 \times 6$ matrix of degree $2$ polynomials in~$R_C$. 
Notice that these quadrics correspond to $Q_{c+1}, \dotsc, Q_{c+6}$ in the proof of our Classification Theorem~\ref{tClassification}.
\end{remark}

\begin{remark}
Let $D \subset \SEbar \subset \PP^{16}$ be a lift of~$C \subset \PP^3$.
Then
\[
	\codim \bigl \langle D \bigr \rangle \ge c+7
\]
where $c$ is the number of independent quadrics in the ideal of~$C$.
Indeed $D \subset \weightedPP$ satisfies $7$ additional equations of the form
\[
	x_1-\sigma^\#(x_1), \dots, r -\sigma^\#(r),
\]
where we chose representatives of $\sigma^\#(x_1), \dotsc, \sigma^\#(r)$ in~$K[e_0, \dotsc, e_3]$. 
Any other choice of representatives differs by a quadric vanishing on~$C$, and hence also on~$D$.
Via the $2$-uple embedding $\alpha \colon \weightedPP \to \PP^{16}$, these equation become linear in the~$\PP^{16}$ containing~$D$.
\end{remark}

\begin{proposition}[Necessary conditions -- algebraic version] 
\label{pNecessaryAlgebraic}
Let $C \subset \PP^3$ be a space curve and 
$\sigma \colon C \to \SEbar \subset \weightedPP$ be a morphism defining a lift of~$C$ to~$\SEbar$. 
Consider the matrices
\[
	N_X := 
	\begin{pmatrix}
		{-{e}_{2}}&
		{e}_{1}&
		{e}_{0}\\
		{-{e}_{3}}&
		{-{e}_{0}}&
		{e}_{1}\\
		{e}_{0}&
		{-{e}_{3}}&
		{e}_{2}\\
		{e}_{1}&
		{e}_{2}&
		{e}_{3}\\
	\end{pmatrix}
	\quad\text{and}\quad
	N_Y := 
	\begin{pmatrix}
		{e}_{2}&
		{-{e}_{1}}&
		{e}_{0}\\
		{e}_{3}&
		{-{e}_{0}}&
		{-{e}_{1}}\\
		{e}_{0}&
		{e}_{3}&
		{-{e}_{2}}\\
		{e}_{1}&
		{e}_{2}&
		{e}_{3}\\
	\end{pmatrix}
\]
and set
\[
	Q_X := \sigma^\#(x_1,x_2,x_3) \quad \text{and} \quad Q_Y := \sigma^\#(y_1,y_2,y_3). 
\]
Then
\[
	(N_X | N_Y) (Q_X | Q_Y)^t \in \bigl(R_C\bigr)^4 \  \text{is the zero vector}
\]
and
\[
	\begin{pmatrix}
		N_X & 0 \\
		0 & N_Y
	\end{pmatrix}
	(Q_X | Q_Y)^t \in \bigl(R_{C^\infty}\bigr)^8 \  \text{is the zero vector,}
\]
where $C^\infty := C \cap \EulerInfty$ with its reduced scheme structure.
\end{proposition}

\begin{proof}
Set $x := (x_1,x_2,x_3)$ and $y := (y_1,y_2,y_3)$. 
Let $P_0,\dots,P_3$ be the degree~$3$ Pfaffians in the ideal of~$\SEbar \subset \weightedPP$ 
defined in Proposition~\ref{pEquationsWeightedPP}. 
A straightforward calculation shows that
\[
	(N_X | N_Y) (x,y)^t = (P_0,P_1,P_2,P_3)^t.
\]
Applying $\sigma^\#$ to this equation gives the first necessary condition.

Each reduced point of~$C^{\infty}$ must be mapped to the border~$\SEinfty$ of~$\SEbar$. 
By Proposition~\ref{pContactH}, the ideal of~$\SEinfty$ contains two sets of four degree~$3$ Pfaffians. 
As above, a straightforward calculation shows that we can write these Pfaffians as
\[
	\begin{pmatrix}
		N_X & 0 \\
		0 & N_Y
	\end{pmatrix}
	(x ,y)^t = (P_{0,0}, \dotsc, P_{0,3}, P_{1,0}, \dotsc, P_{1,3})^t \,.
\]
Over each reduced point of~$C \cap \EulerInfty$ the application of~$(\sigma|_{C^\infty})^\#$
to these equations must vanish. This gives the second necessary condition.
\end{proof}

\begin{remark}
Notice that the above necessary conditions are linear in the coefficients of~$(Q_X | Q_Y)$.
\end{remark}

We will now give a geometric interpretation of these conditions. 
The key point is the following

\begin{lemma} \label{lKoszul}
After a suitable change of basis, the $4 \times 6$ matrix $(N_X | N_Y)$ appears in the Koszul complex associated
to $(e_0,\dots,e_3)$ as follows
\[
	0 \to \sOPthree \to 4 \sOPthree(1) \to 6 \sOPthree(2) \xrightarrow{(N_X | N_Y)} 4 \sOPthree(3) \to \sOPthree(4) \to 0.
\]
In particular we have an exact sequence
\[
	0 \to \TPthree \to 6 \sOPthree(2) \xrightarrow{(N_X | N_Y)} 4 \sOPthree(3) \to \sOPthree(4) \to 0
\]
where $\TPthree$ is the tangent bundle on~$\PP^3$.
\end{lemma}

\begin{proof}
The first claim is an easy calculation, which for example can be done via computer algebra. 
The second claim follows since the kernel~$\sK$ of~$(N_X | N_Y)$ 
is the same as the cokernel of the first map in the complex, i.e.
\[
0 \to \sOPthree \to 4 \sOPthree(1) \to \sK \to 0.
\]
Now the first map is the same as in the Euler sequence on~$\PP^3$ and 
therefore $\sK$ must be the tangent bundle~$\TPthree$.
\end{proof}

\begin{corollary} \label{cMatrixToSection}
Let $C \subset \PP^3$ be a subvariety. 
Then we have a $1:1$ correspondence between 
matrices of $6$ quadrics $(Q_X | Q_Y)$ with
\[
	(N_X | N_Y) (Q_X | Q_Y)^t \in \bigl(R_C\bigr)^4 \  \text{is the zero vector}
\]
and sections $\tau \in H^0(\TPthree|_C)$ via the following diagram
\xycenter{
	&& \sO_C \ar@{-->}[dl] _{\tau} \ar[d]^{(Q_X | Q_Y)^t}\\
	0 \ar[r] &\TPthree|_C \ar[r] & 6\sO_C(2) \ar[rr]^{(N_X | N_Y)} && 4 \sO_C(3).
}
\end{corollary}

\begin{proof}
Since the sequences of Lemma~\ref{lKoszul} are sequences of vector bundles, 
they remain exact after tensoring with $\sO_C$. 
This shows that the bottom row of the above diagram is exact.

For each~$\tau$ we obtain a $(Q_X | Q_Y)^t$ by composing with $\TPthree \to 6\sO(2)$. 
It satisfies $(N_X | N_Y) (Q_X | Q_Y)^t = 0$ since the sequence at the bottom is exact.

Conversely, a map defined by~$(Q_X | Q_Y)^t$ lifts to a section~$\tau$ if it maps to the kernel of~$(N_X | N_Y)$,
i.e., if it satisfies
$
	(N_X | N_Y) (Q_X | Q_Y)^t = 0.
$
\end{proof}

\begin{proposition}[Necessary conditions -- geometric version] 
\label{pNecessaryGeometric}
Let $C \subset \PP^3$ be a space curve. 
There is a $1:1$ correspondence between 
matrices $(Q_X | Q_Y)$ satisfying the two algebraic conditions in Proposition~\ref{pNecessaryAlgebraic} and 
sections $\tau \in H^0(\TPthree|_C)$ that lift to a section~$\tau'$ of~$T_{\EulerInfty}$ over~$C^\infty$ 
via the normal bundle sequence of~$\EulerInfty$ restricted to~$C^\infty$:
\xycenter{
	&& \sO_{C^\infty} \ar@{-->}[dl] _{\tau'} \ar[d]^{\tau}\\
	0 \ar[r] &T_{\EulerInfty}|_{C^\infty} \ar[r] & \TPthree|_{C^{\infty}} \ar[r] & N_{\EulerInfty/\PP^3}|_{C^\infty} \ar[r] & 0 \,.
}
\end{proposition}

\begin{proof}
Since $K$ is algebraically closed, we have an isomorphism
\[
	\PP^1 \times \PP^1 \cong \EulerInfty \subset \PP^3 \,.
\]
Over this $\PP^1 \times \PP^1$ we then have an exact sequence
\[
	0 
	\to
	T_{\PP^1 \times \PP^1}
	 \to 
	6 \sOPii(2,2) 
	\xrightarrow{
		\begin{pmatrix}
			N_X & 0 \\
			0 & N_Y
		\end{pmatrix}
	}
	8 \sOPii(3,3) \,,
\]
which can be checked with a computer algebra program. 
In fact, notice that $T_{\PP^1 \times \PP^1} \cong \sO_{\PP^1 \times \PP^1}(2,0) \oplus \sO_{\PP^1 \times \PP^1}(0,2)$,
since $T_{\PP^1 \times \PP^1} \cong \rho_1^{\ast} T_{\PP^1} \oplus \rho_2^{\ast} T_{\PP^1}$, 
where $\rho_i \colon \PP^1 \times \PP^1 \longrightarrow \PP^1$ is the projection on the $i$-th component.
Then, by parametrizing $\EulerInfty$ via a map $K[s_0, s_i] \otimes K[t_0, t_i] \longrightarrow K[e_0, e_1, e_2, e_3]$,
one sees via a direct computer algebra computation that the kernel of~$N_X$ is isomorphic to $\sO_{\PP^1 \times \PP^1}(2,0)$,
while the kernel of~$N_Y$ is isomorphic to $\sO_{\PP^1 \times \PP^1}(0,2)$.

Now we restrict the previous exact sequence to~$C^{\infty}$ and we consider the following diagram:
\xycenter{
	&& \sO_{C^{\infty}} \ar@{-->}[dl] _{\tau'} \ar[d]^{(Q_X | Q_Y)^t} \\
	0 \ar[r] & T_{\EulerInfty}|_{C^{\infty}} \ar[r] & 6\sO_{C^{\infty}}(2) \ar[rr]_{\begin{pmatrix}
	N_X & 0 \\
	0 & N_Y
	\end{pmatrix}} && 8 \sO_{C^{\infty}}(3)
}
Arguing as in the proof of Corollary~\ref{cMatrixToSection} we see that 
the second algebraic condition in Proposition~\ref{pNecessaryAlgebraic} is equivalent to the existence of the section~$\tau'$.

We are left to check that the sections~$\tau$ and~$\tau'$ are compatible in the normal bundle sequence of the statement. 
This is true because on~$C^{\infty}$ the second algebraic condition in Proposition~\ref{pNecessaryAlgebraic} implies the first one, hence we have the diagram:
\xycenter{
	& 0 \ar[d] \\
	0 \ar[r] & T_{\EulerInfty}|_{C^{\infty}} \ar[d] \ar[r] & 6\sO_{C^{\infty}}(2) \ar[d]^{=} \ar[rr]^{\begin{pmatrix}
	N_X & 0 \\
	0 & N_Y
	\end{pmatrix}} && 8 \sO_{C^{\infty}}(3) \\
	0 \ar[r] &\TPthree|_{C^{\infty}} \ar[d] \ar[r] & 6\sO_{C^{\infty}}(2) \ar[rr]^{(N_X | N_Y)} && 4 \sO_{C^{\infty}}(3) \\
	& N_{\EulerInfty/\PP^3}|_{C^\infty} \ar[d] \\
	& 0
}
So indeed $\tau'$ is a lift of~$\tau$ as claimed.

Conversely, if $\tau$ can be lifted to~$\tau'$, then the matrix~$(Q_X | Q_Y)$ satisfies the second algebraic condition in Proposition~\ref{pNecessaryAlgebraic}.
\end{proof}

Surprisingly, the conditions above are often also sufficient:

\begin{theorem} \label{tConstruction}
Let $C \subset \PP^3$ be a space curve satisfying the following conditions:
\begin{enumerate}
\item[(A)] $C$ is quadratically normal, i.e., the natural map $H^0 \bigl( \sOPthree(2) \bigr) \to H^0 \bigl( O_C(2) \bigr)$ is surjective;
\item[(B)] the intersection multiplicity of~$C \cap \EulerInfty$ is at most~$2$ at each reduced point. 
\end{enumerate}
Then we have a $1:1$ correspondence between
\begin{enumerate}
\item \label{item:lifts} lifts $D \subset \SEbar$ as in Definition~\ref{dLift}, 
\item \label{item:matrices} matrices $(Q_1,\dots,Q_6)$ of quadrics in~$R_C$ satisfying the conditions of Proposition~\ref{pNecessaryAlgebraic}, and
\item \label{item:sections} sections $\tau$ of~$\TPthree|_C$ satisfying the conditions of Proposition~\ref{pNecessaryGeometric}.
\end{enumerate}
\end{theorem}

\begin{proof}
We have seen so far that \eqref{item:lifts} implies \eqref{item:matrices} by Proposition~\ref{pNecessaryAlgebraic}, 
and \eqref{item:matrices} is equivalent to \eqref{item:sections} by Proposition~\ref{pNecessaryGeometric}.

It remains therefore to prove that each matrix $(Q_X | Q_Y) =: (Q_1,\dots,Q_6)$ 
satisfying the conditions of Proposition~\ref{pNecessaryAlgebraic} gives rise to a lift~$D$ of~$C$. 
To do this, we construct the corresponding graded ring homomorphism~$\sigma^\#$. 
More precisely, we need to find a quadric~$Q_7 \in R_C$ such that the image~$D$ of the map~$\sigma \colon C \to \weightedPP$ defined by
\[
	\sigma^\#(e_0, \dotsc, e_3, x_1, x_2, x_3, y_1, y_2, y_3, r) := (e_0, \dotsc, e_3, Q_1, \dotsc, Q_6, Q_7) ,
\]
with $Q_1,\dots,Q_6$ as above, is inside~$\SEbar$. 
By the proof of Proposition~\ref{pNecessaryAlgebraic}, we know that the first condition of Proposition~\ref{pNecessaryAlgebraic} 
is equivalent to the fact that all the four degree~$3$ Pfaffians in the ideal of~$\SEbar$ vanish on~$D$. 
(Notice that these equations do not involve the variable~$r$). 
Since $D$ is integral, by Proposition~\ref{pDeg3impliesDeg4} all the Pfaffians in the ideal of~$\SEbar$ vanish on~$D$. 
Therefore, all equations of~$\SEbar$ not involving the variable~$r$ vanish on~$D$ regardless of the choice of~$Q_7$.
By Proposition~\ref{pEquationsWeightedPP}, the only polynomial we have to check is therefore
\[
	y_1^2+y_2^2+y_3^2-rh .
\]
We apply $\sigma^{\#}$ to this equation, and obtain
\begin{align*}
	0 = \sigma^\#(y_1^2+y_2^2+y_3^2-rh) &\iff
	\sigma^\#(y_1^2+y_2^2+y_3^2)=\sigma^\#(r) \sigma^\#(h) \\
	&\iff Q_4^2 + Q_5^2 + Q_6^2 = Q_7 (e_0^2 + e_1^2 + e_2^2 + e_3^2) .
\end{align*}
This equation has a solution for~$Q_7$ if and only if $e_0^2 + \dots + e_3^2$ divides $Q_4^2 + Q_5^2 + Q_6^2$ in~$R_C$.
This is equivalent to $Q_4^2 + Q_5^2 + Q_6^2 \in \bigl(e_0^2 + \dots + e_3^2\bigr) \subset R_C$ or 
\[
	V \bigl( e_0^2 + \dots + e_3^2 \bigr) \subset 
	V \bigl( Q_4^2 + Q_5^2 + Q_6^2 \bigr) \,,
\]
where both vanishing sets are considered with their induced, possibly nonreduced, scheme structure of subschemes of~$C$.
Notice that $V \bigl( e_0^2 + \dots + e_3^2 \bigr)$ is $C \cap \EulerInfty$ (considered as a scheme).

By assumption (B), the scheme~$C \cap \EulerInfty$ consist of points with multiplicity~$1$ or~$2$.
Let $p + \epsilon p'$ with $\epsilon^2 = 0$ be such a point, 
so $p'=0$ if and only if the point has multiplicity~$1$. 
We need to prove that
\[
	(Q_4^2 + Q_5^2 + Q_6^2)(p + \epsilon p') = 0.
\]
For this we define $q + \epsilon q'$ by
\[
	q_j + \epsilon q_j' := Q_j(p+\epsilon p') 
	\quad \text{for } j \in \{ 1, \dotsc, 6 \},
\]
i.e.
\[
	\sigma(p+ \epsilon p') = 
	(p,q,*) + \epsilon (p',q',*)
\]
with the last entry unknown so far.

Since $p+\epsilon p'$ lies on $C$, we can restrict first condition of Proposition~\ref{pNecessaryAlgebraic}
further to $p + \epsilon p'$ and obtain
\begin{equation}
\tag{i}
\label{condition1}
	\bigl(
		N_X(p+\epsilon p') \,| \, N_Y(p+\epsilon p')
	\bigr)
	(q + \epsilon q')^t = 0.
\end{equation}
Since $p$ is a reduced point and therefore in $C^\infty$, 
we can restrict the second condition of Proposition~\ref{pNecessaryAlgebraic} further to~$p$ and obtain
\begin{equation}
\tag{ii}
\label{condition2}
	\begin{pmatrix}
	N_X(p) & 0 \\
	0 & N_Y(p) 
	\end{pmatrix} 
	q^t = 0. 
\end{equation}
Finally, $p+\epsilon p'$ lies on $\EulerInfty = V(e_0^2+e_1^2+e_2^2+e_3^2)$ and therefore
\begin{equation}
\tag{iii}
\label{condition3}
	(p_0+\epsilon p_0')^2 + \dots + (p_3+\epsilon p_3)^2 = 0.
\end{equation}
Now consider $p_i,p_i'$ for $i \in \{0,\dots, 3\}$ and $q_j,q_j'$ for $j \in \{1,\dots,6\}$ as
free variables. The three conditions above then are equivalent to 
the vanishing of sets of polynomials in these variables.
Let $I \subset K[p_0,\dots,p_3,q_1,\dots,q_6]$ be the ideal generated by these polynomials. 
Furthermore, for $i \in \{ 0, \dots, 3 \}$, let $I_i$ be the ideal defined by the condition
\begin{equation}
\tag{iv}
\label{condition4}
	p_i^2\bigl((q_4+\epsilon q_4)^2+(q_5+\epsilon q_5)^2+(q_6+\epsilon q_6)^2\bigr) = 0.
\end{equation}
Now, with a computer algebra program, we can check that
\[
	 I_i \subset I
	\quad \text{for $i \in \{0,\dots,3\}$}.
\]
This implies that every $p+\epsilon p'$ that satisfies conditions \eqref{condition1}, \eqref{condition2}, and \eqref{condition3} 
also satisfies condition~\eqref{condition4} for all $i \in \{0,\dots,3\}$.

In particular this is true for the point of multiplicity at most $2$ that we started with. 
Here at least one of the~$p_i$ is nonzero, so we obtain
\[
	0 = (q_4 + \epsilon q_4)^2 + (q_5 + \epsilon q_5)^2 + (q_6+\epsilon q_6)^2
	  = (Q_4^2 + Q_5^2 + Q_6^2)(p + \epsilon p').
\]
So we proved that $C \cap \EulerInfty \subset V \bigl( Q_4^2 + Q_5^2 + Q_6^2 \bigr)$, 
hence $Q_7$ can be chosen such that $D = \sigma(C)$ is inside~$\SEbar$, and $D$ is a lift of~$C$.
\end{proof}

Now that we established the $1 \colon 1$ correspondence between lifts, $1 \times 6$ matrices,
and sections of the tangent bundle as in Theorem~\ref{tConstruction},
we can aim at constructing families of lifts of the curves in Table~\ref{tWithQuadrics} and~\ref{tWithoutQuadrics}
and computing their dimension.
We start by calculating the cohomology of~$\TPthree|_C$ in two ways:

\begin{lemma} \label{lSectionsLarged}
Let $C \subset \PP^3$ be a smooth space curve of degree~$d$ and genus~$g$ with $d > 2g-2$. 
Then 
\[
	h^0(\TPthree|_C) = 4d+3-3g \quad \text{and} \quad h^1(\TPthree|_C) = 0 \,.
\]
\end{lemma}

\begin{proof}
We consider the Euler sequence
\[
	0 \to \sOPthree \to 4 \sOPthree(1) \to \TPthree \to 0.
\]
Since this is an exact sequence of vector bundles, it remains exact if we tensor with~$\sO_C$:
\[
	0 \to \sO_C\to 4 \sO_C(1) \to \TPthree|_C \to 0.
\]
Since $d > 2g-2$, we have $H^1(4 \sO_C(1)) = 0$. 
Therefore we obtain the following long exact sequence of cohomology groups:
\[
	\xymatrix@R=.2cm{
	0 \ar[r] & H^0(\sO_C) \ar[r]& H^0\bigl(4\sO_C(1)\bigr) \ar[r]& H^0(\TPthree|_C) \ar[r]& \\
	\mbox{ } \ar[r]& H^1(\sO_C) \ar[r]& 0 \ar[r]&  H^1(\TPthree|_C) \ar[r]&  0
	}
\]
Consequently
\begin{align*}
	h^0(\TPthree|_C) 
	&= h^0 \bigl(4 \sO_C(1) \bigr) - \bigl(h^0(\sO_C)-h^1(\sO_C) \bigr) \\
	&= 4(d+1-g) - (1-g) \\
	&= 4d+3-3g
\end{align*}
and $h^1(\TPthree|_C)=0$.
\end{proof}

\begin{lemma} \label{lTMFR}
Let $C \subset \PP^3$ be a space curve with Betti table of the form
\[
	\begin{matrix}
		1 & - & - & - \\
		- & \beta_{1,2} & \beta_{2,3} & -  \\
		- & \beta_{1,3} & \beta_{2,4} & \beta_{3,5} \\
		- & \beta_{1,4} & \beta_{2,5} & -
	\end{matrix}
\]
with $\beta_{1,4}\beta_{2,4}=0$. Then
\[
	h^0(\TPthree|_C) = 15+\beta_{2,4} \quad \text{and} \quad h^1(\TPthree|_C) = \beta_{1,4} \,.
\]
\end{lemma}

\begin{proof}
By the assumptions, $C$ has a minimal free resolution
\[
	0 \from \sO_C \from
	\sOPthree 
	\from
	\underbrace{ 
	\begin{matrix}
		\beta_{1,2} \, \sOPthree(-2) \\
		\oplus \\
		\beta_{1,3} \, \sOPthree(-3) \\
		\oplus \\
		\beta_{1,4}  \, \sOPthree(-4)
	\end{matrix}}_{E_1}
	\from
	\underbrace{
	\begin{matrix}
		\beta_{2,3} \, \sOPthree(-3) \\
		\oplus \\
		\beta_{2,4} \, \sOPthree(-4) \\
		\oplus \\
		\beta_{2,5} \, \sOPthree(-5)
	\end{matrix}}_{E_2}
	\from
	\underbrace{
	\begin{matrix}
		\\
		\\
		\beta_{3,5} \, \sOPthree(-5) \\
		\\
		\\
	\end{matrix}}_{E_3}
	\from 
	0 \,.
\]
Tensoring with~$\TPthree$ is an exact functor since $\TPthree$ is a vector bundle. 
We obtain an exact sequence
\[
	0 \from \TPthree|_C \from \TPthree 
	\from E_1 \otimes \TPthree \from E_2 \otimes \TPthree \from E_3 \otimes \TPthree \from 0 \,.
\]
From the Euler sequence we obtain that $h^i \bigl( \TPthree(-j) \bigr)$ vanishes
for all $i = 0\dots3$ and all $j = 2\dots 5$ except for
\[
	h^2 \bigl(\TPthree (-4) \bigr) = 1 \,.
\]
It follows that in the exact sequence above the only nonzero cohomology in the sequence above is
\begin{align*}
	h^0(\TPthree) &= 15 \quad \text{(also from the Euler sequence)} \,, \\
	h^2(E_1 \otimes \TPthree) &= \beta_{1,4} \,, \\
	h^2(E_2 \otimes \TPthree) &= \beta_{2,4} \,.
\end{align*}
We now split the above resolution into short exact sequences:
\begin{enumerate}
\item $0 \from K_2 \from E_2 \otimes \TPthree \from E_3 \otimes \TPthree \from 0$ \,, \label{split1}
\item $0 \from K_1 \from E_1 \otimes \TPthree \from K_2 \from 0$ \,, \label{split2}
\item $0 \from \TPthree|_C \from \TPthree \from K_1 \from 0$ \,. \label{split3}
\end{enumerate}
From \eqref{split1} we obtain 
\[
	h^2(K_2) = h^2(E_2 \otimes \TPthree) = \beta_{2,4} \,.
\]
Sequence \eqref{split2} gives
\[
	0 \from H^2(K_1) \from H^2(E_1 \otimes \TPthree) \xleftarrow{\alpha} H^2(K_2) \from H^1(K_1) \from 0 \,.
\]
Since either $h^2(E_1 \otimes \TPthree) = \beta_{1,4}$ or $h^2(K_2) = \beta_{2,4}$ is zero, the map~$\alpha$
must also be zero. Therefore
\[
	h^2(K_1) = \beta_{1,4} \quad \text{and} \quad h^1(K_1) = \beta_{2,4} \,.
\]
Since $H^1(\TPthree) = 0$, sequence~\eqref{split3} finally gives
\[
	0 \from H^2(K_1) \from H^1(\TPthree|_C) \from 0 \from H^1(K_1) \from H^0(\TPthree|_C) \from H^0(\TPthree) \from 0
\]
and therefore
\[
	h^0(\TPthree|_C) = 15+\beta_{2,4} \quad \text{and} \quad h^1(\TPthree|_C) = \beta_{1,4} \,. \qedhere
\]
\end{proof}

From this we also get the cohomology of the normal bundle for our space curves. 
This is useful to calculate the dimension of the Hilbert scheme of curves.

\begin{proposition} 
\label{pHilbP3}
Let $C \subset \PP^3$ be a smooth space curve of degree~$d$ and genus~$g$, such that
\[
	h^1(\TPthree|_C) = 0 \,.
\]
Let $\Hilb_C$ be the component of the Hilbert scheme of degree~$d$ and genus~$g$ space curves that contains~$C$. 
Then $\Hilb_C$ is smooth at~$C$ and has dimension
\[
	h^0(N_{C/\PP^3}) = h^0(\TPthree|_C) + 3g - 3 = 4d.
\]
\end{proposition} 
\begin{proof}
We consider the normal bundle sequence 
\[
	\xymatrix{
	0 \ar[r] & T_C \ar[r] & \TPthree|_C \ar[r] & N_{C/\PP^3} \ar[r] & 0
	} \,.
\]
Since $h^1(\TPthree|_C)=0$ we also get $h^1(N_{C/\PP^3}) = 0$ and
\begin{align*}
	h^0(N_{C/\PP^3}) 
	&= h^0(\TPthree|_C) - h^0(T_C) + h^1(T_C) \\ 
	&= h^0(\TPthree|_C) - (2-2g + 1 - g) \\
	& = h^0(\TPthree|_C) + 3g - 3 \,.
\end{align*}
It is known that $H^0(N_{C/\PP^3})$ is the tangent space of~$\Hilb_C$ at~$C$ 
(see \cite[Theorem~1.1~(b)]{Hartshorne2010}). 
Since $H^1(N_{C/\PP^3}) = 0$, the deformation is unobstructed, i.e., 
$\Hilb_C$ is smooth at~$C$ and therefore of the claimed dimension 
(see \cite[Theorem~1.1~(c)]{Hartshorne2010}).
The second equality comes from $h^0(\TPthree|_C) = 4d + 3 -3g$ as in the proof of Lemma~\ref{lSectionsLarged}, 
where only $h^1(\TPthree|_C) = 0$ is used.
\end{proof}

Before stating the main theorem of this section, 
recall that inversion bonds are tangential intersections of a configuration curve with the boundary, 
while butterfly bonds are simple intersections. 
Therefore, if we have a curve $D \subseteq \SEbar$ of degree~$2d$ with no similarity, collinearity, or vertex bonds,
and $d'$ is the number of inversion bonds, then the number of butterfly bonds will be $2d-2d'$.

\begin{theorem} \label{tFamilies}
Pick a triple $(c,d,g)$ from Table~\ref{tWithQuadrics} and Table~\ref{tWithoutQuadrics} not marked by~$(\ast)$. 
Let $\mathrm{Hilb}_{\PP^{16}}(2d,g)$ be the Hilbert scheme of curves of degree~$2d$ and genus~$g$ in~$\PP^{16}$.
Let $H(c,d',d) \subset \mathrm{Hilb}_{\PP^{16}}(2d,g)$ be the locally closed subscheme of curves $D \in \mathrm{Hilb}_{\PP^{16}}(2d,g)$ such that:
\begin{enumerate}
\item $D \subset \SEbar \subset \PP^{16}$,
\item $\codim \langle D \rangle = c+7$, 
\item $D$ does not have similarity, collinearity, or vertex bonds, 
\item $D$ has $d'$ inversion bonds and $2d-2d'$ butterfly bonds, and
\item the restriction of the projection $\pi \colon \SEbar \dashrightarrow \PP^3$ to~$D$ defines an isomorphism to the image.
\end{enumerate}
Then $H(c,d',d)$ has at least one irreducible component~$H'$ with dimension at least
\[
	h^0(\TPthree|_{D}) + 2d. 
\]
These expected dimensions are listed in Table~\ref{tWithQuadricsDimensions} and Table~\ref{tWithoutQuadricsDimensions}. 
\end{theorem}

\begin{proof}
First, we prove that $H(c,d',d)$ is not empty. 
Via computer algebra, we construct, over a finite field~$\FF_p$, 
a \emph{smooth} curve~$C_p \subset \PP^3$ of type $(c,d,g)$ in Table~\ref{tWithQuadrics} or Table~\ref{tWithoutQuadrics}
that intersects~$\EulerInfty$ in $d'$ distinct double points and $2d - 2d'$ distinct simple points. 
By Proposition~\ref{pExistsCharZero}, this implies the existence of a curve~$C$ with the same properties also over~$\CC$.
By Theorem~\ref{tConstruction}, lifts~$D$ of~$C$ that satisfy all conditions of the theorem 
are in bijection with elements of~$H^0(\TPthree|_{C})$ that lift to~$H^0(T_{\EulerInfty}|_{C^\infty})$. 
Imposing to lift to~$H^0(T_{\EulerInfty}|_{C^\infty})$ 
gives at most one condition to elements of~$H^0(\TPthree|_{C})$ for each reduced intersection point 
(namely, being tangent to a hypersurface), 
i.e., at most $d'+ (2d-2d') = 2d-d'$ conditions. 
The dimension of the space of elements ~$H^0(\TPthree|_{C})$ that lift to~$H^0(T_{\EulerInfty}|_{C^\infty})$ is then at least 
\[
	h^0(\TPthree|_{C}) - 2d + d' \geq h^0(\TPthree|_{C}) - 2d \,.
\]
Looking at Table~\ref{tWithQuadricsDimensions} and Table~\ref{tWithoutQuadricsDimensions}, 
we see that this quantity is always at least~$2$. 
Hence lifts~$D$ of~$C$ with the conditions from the statement always exist, i.e., $H(c,d',d)$ is not empty.

We consider a component~$H'$ of~$H(c,d',d)$ that contains curves~$D$ with smooth~$\pi(D)$. 
Let $\mathrm{Hilb}_{\PP^3}(d,g)$ be the Hilbert scheme of smooth curves in $\PP^3$ of degree~$d$ and genus~$g$. 
Since by hypothesis $\pi|_D$ is an isomorphism, we get a rational map $H' \dashrightarrow \mathrm{Hilb}_{\PP^3}(d,g)$.
Let $Y'$ be the image of~$H'$ under this map.

Since the elements $C \in \mathrm{Hilb}_{\PP^3}(d,g)$ are smooth, 
we can apply Lemma~\ref{lSectionsLarged} and Lemma~\ref{lTMFR} 
to see that $h^1(\TPthree|_{\pi(D)}) = 0$ in all cases not marked with~$(\ast)$. 
Proposition~\ref{pHilbP3} then implies $\dim \mathrm{Hilb}_{\PP^3}(d,g) = 4d$.

If $P$ is an inversion bond of~$D$, 
then $\pi(D)$ is contact to~$\EulerInfty$ by Corollary~\ref{cContactButterfly} at~$\pi(P)$. 
This imposes $d'$ conditions on the choice of~$\pi(D)$. 
Therefore, we have
\[
	\dim Y' \ge 4d -d' = 4d - d'.
\]
Again, by Theorem~\ref{tConstruction} the fibers of the natural projection
\[
	H' \to Y'
\]
have dimension at least $h^0(\TPthree|_{C}) - 2d+d'$. 
Therefore $H'$ has dimension at least 
\begin{align*}
	\dim Y' + h^0(\TPthree|_{C}) - 2d+d' &\ge 
	4d - d' + h^0(\TPthree|_{C}) - 2d+d' \\
	&= h^0(\TPthree|_{C}) + 2d. \qedhere
\end{align*}
\end{proof}

\begin{table}
\caption{Families of lifts of Euler curves as in Theorem~\ref{tFamilies}, 
and expected dimensions of the corresponding families of mobile hexapods. 
In the cases marked by~($\ast$), 
the scheme~$X$ from Proposition~\ref{pClassificationAlgebraic} must be a complete intersection of a quadric and a quartic surface in~$\PP^3$.}
\begin{tabular}{ccc|c|ccccc}
\toprule
 $c$ & $d$ & $g$& Betti table of~$C$ & $h^0(\TPthree|_C)$ & $h^1(\TPthree|_C)$ & $h^0(N_{C/\PP^3})$ & $\dim \mathrm{Hilb}$ & $\mathrm{edim} \, \text{$6$-pods}$ \\ 
\midrule
 $3$ & 3 & 0 & $\begin{matrix} 1 & - & - \\  - & 3 & 2  \end{matrix}$ & 15 & 0 & 12 & 21 & $\ge 16$\\
 \midrule
 $2$ & 4 & 1 & $\begin{matrix} 1 & - & - \\  - & 2& - \\ - & - & 1  \end{matrix}$ & 16 & 0 & 16 & 24 & $\ge 13$ \\
 \midrule
 $1$ & 4 & 0 & $\begin{matrix} 1 & - & -  & - \\  - & 1 & - & - \\ - & 3 & 4 & 1  \end{matrix}$ &  19 & 0 & 16 & 27 & $\ge 10$\\
 \midrule
	& 5 & 2 & $\begin{matrix} 1 & - & -  \\  - & 1 & -  \\ - & 2 & 2  \end{matrix}$ & 17 & 0 & 20 & 27 & $\ge 10$\\
 \midrule
 	& 6 & 4 & $\begin{matrix} 1 & - & -  \\  - & 1 & -  \\ - & 1 & - \\ - & - & 1  \end{matrix} $ & 15 & 0 & 24 & 27 & $\ge 10$\\
 \midrule
 	& 5 & 0 & $\begin{matrix} 1 & - & -  & - \\  - & 1 & - & - \\ - & - & - & - \\ - & 4 & 6 & 2  \end{matrix} \quad \hfill (\ast)$ & 23 & 0 & 20 & 33 & $\ge 16$\\ \hline
	& 6 & 3 & $\begin{matrix} 1 & - & -  & - \\  - & 1 & - & - \\ - & - & - & - \\ - & 3 & 4 & 1 \end{matrix} \hfill (\ast)$  & 18  & 0 & 24 & 30 & $\ge13$\\ \hline
	& 7 & 6 & $\hfill \begin{matrix} 1 & - & -  \\  - & 1 & -  \\ - & - & - \\ - & 2 & 2  \end{matrix} \hfill (\ast)$ & 15 & 2  \\
 \midrule
	& 8 & 9 & $\hfill \begin{matrix} 1 & - & -  \\  - & 1 & -  \\ - & - & -  \\ - & 1 & - \\ - & - & 1\end{matrix} \hfill (\ast)$ & \\
 \bottomrule
\end{tabular}
\label{tWithQuadricsDimensions}
\end{table} 

\begin{table}
\caption{Families of lifts of Euler curves as in Theorem~\ref{tFamilies}, 
and expected dimensions of the corresponding families of mobile hexapods.}
\begin{tabular}{ccc|c|ccccccc}
\toprule
$c$ & $d$ & $g$& Betti table of~$C$ & $h^0(\TPthree|_C)$ & $h^1(\TPthree|_C)$ & $h^0(N_{C/\PP^3})$ & $\dim \mathrm{Hilb}$ & $\mathrm{edim} \, \text{$6$-pods}$ \\ 
 \midrule
 $0$ & 5 & 0 & $\begin{matrix} 1 & - & -  & - \\  - & - & - & - \\ - & 4 & 3 & - \\ - & 1 & 2 & 1  \end{matrix}$ & 23 & 0 & 20 & 33 & $\ge 10$\\
 \midrule
     & 5 & 1 & $\begin{matrix} 1 & - & -  & - \\  - & - & - & - \\ - & 5 & 5 & 1  \end{matrix}$ & 20 & 0 & 20 & 30 & $\ge7$ \\
 \midrule
     & 6 & 3 & $\begin{matrix} 1 & - & -   \\  - & - & -  \\ - & 4 & 3  \end{matrix}$ & 18 & 0 & 24 & 30 & $\ge 7$\\
 \midrule
     & 6 & 2 & $\begin{matrix} 1 & - & -  & - \\  - & - & - & - \\ - & 3 & 1 & - \\ - & 1 & 3 & 1  \end{matrix}$ &  21 & 0  & 24  & 33  & $\ge 10$\\
 \midrule
     & 7 & 5 & $\begin{matrix} 1 & - & -   \\  - & - & -  \\ - & 3 & 1 \\ - & - & 1  \end{matrix}$ & 16 & 0 & 28 & 30 & $\ge 7$ \\
 \bottomrule
 \end{tabular}
 \label{tWithoutQuadricsDimensions}
\end{table}

\begin{corollary} \label{cFamilyGeneralizedHexapods}
Let $H$ be as in Theorem~\ref{tFamilies} and let $H' \subset H$ be the subfamily of generalized hexapod curves. 
Then the expected dimension of~$H'$ is at least 
\[
	 \dim H - 5 - 6(3-c).
\]
These expected dimensions are all positive and listed in Table~\ref{tWithQuadricsDimensions} and Table~\ref{tWithoutQuadricsDimensions}.
\end{corollary}

\begin{proof}
A generalized hexapod curve must pass through the identity point of~$\SE_3$. 
Since $\dim \SE_3 = 6$, this imposes at most $5$ conditions.
 
By Proposition~\ref{pDimensionCount}, since the codimension of the span of these curves is~$c+7$, we have
\[
	\dim H' \ge (\dim H -5) - 6\bigl(10-(c+7)\bigr) \ge \dim H - 5 - 6(3-c).
\]
This gives the values in the last column of Table~\ref{tWithQuadricsDimensions} and Table~\ref{tWithoutQuadricsDimensions}.
\end{proof}

\begin{remark} \label{rButterflyComponents}
Notice that the bound on the dimension of~$H$ is independent of the number of butterfly points. 
This suggests that for each number of butterfly points the corresponding generalized hexapod curves in~$\SE_3$ form a family of the same dimension, and neither family is a specialization of one of the others. 
\end{remark}

\appendix
\section{Betti tables for nondegenerate irreducible reduced space curves of degree at most \texorpdfstring{$6$}{6}}
\label{sBetti}

Let $C \subset \PP^3$ be a reduced, nondegenerate scheme of pure dimension~$1$, degree~$d$, and arithmetic genus~$g_a$. 
Let
\[
	0 \from I_C 
	\from \bigoplus_{i\ge2} \sOPthree(-i)^{a_i}
	\from \bigoplus_{j\ge3} \sOPthree(-j)^{b_j}
	\from \bigoplus_{k\ge4} \sOPthree(-k)^{c_k}
	\from 0
\]
be the minimal free resolution of the ideal sheaf of~$C$. The curve~$C$ then has Betti table of the form
\[
	\begin{matrix}
		1 & - & - & - \\
		- & a_2 & b_3 & c_4 \\
		- & \vdots & \vdots & \vdots \\
		- & a_n & b_{n+1} & c_{n+2} \\
	\end{matrix}
\]
We call $(a_2,\dots,a_n \,\mid\, b_3,\dots,b_{n+1} \,\mid\, c_4,\dots c_{n+2})$ the \emph{Betti sequence} of~$C$.

In \cite{deg6resolutions} Hirotsugo Naito classifies possible Betti tables 
for nondegenerate, irreducible, \emph{smooth} curves~$C$ of degree~$d$ and geometric genus~$g$:

\begin{theorem}[Naito]
If $d \le 6$, then the Betti sequence of~$C$ is as in Table~\ref{table:naito}.
\begin{table}[ht]
\caption{Possible Betti sequences of smooth integral nondegenerate curves in~$\PP^3$.}
\label{table:naito}
\begin{tabular}{cl|cl}
\toprule
	$(d,g)$ & Betti sequence & $(d,g)$ & Betti sequence \\ \midrule
	$(3,0)$ & $(3\,|\,2\,|\,0)$ & $(6,0)$ & $(1,0,0,5\,|\,0,0,0,8\,|\,0,0,0,3)$ \\ \midrule
	$(4,0)$ & $(1,3\,|\,0,4\,|\,0,1)$ && $(0,1,6\,|\,0,0,9\,|\,0,0,3)$ \\ \midrule
	$(4,1)$ & $(2,0\,|\,0,1\,|\,0,0)$ && $(0,2,2,1\,|\,0,0,4,2\,|\,0,0,1,1)$ \\ \midrule
	$(5,0)$ & $(1,0,4\,|\,0,0,6\,|\,0,0,2)$ & $(6,1)$& $(0,2,3\,|\,0,0,6\,|\,0,0,2)$ \\ \midrule
					& $(0,4,1\,|\,0,3,2\,|\,0,0,1)$ & $(6,2)$& $(0,3,1\,|\,0,1,3\,|\,0,0,1)$ \\ \midrule
	$(5,1)$ & $(0,5\,|\,0,5\,|\,0,1)$ & $(6,3)$& $(1,0,3\,|\,0,0,4\,|\,0,0,1)$ \\ \midrule
	$(5,2)$ & $(1,2\,|\,0,2\,|\,0,0)$ && $(0,4\,|\,0,3\,|\,0,0)$  \\ \midrule
					&& $(6,4)$& $(1,1,0\,|\,0,0,1\,|\,0,0,0)$ \\ 
\bottomrule
\end{tabular}
\end{table}
\end{theorem}
The same classification is true if $C$ is only nondegenerate, irreducible, and reduced.
For this, Naito's proof must only be slightly modified:

Theorem~$2$ in his paper (by Gruson-Lazarsfeld-Peskine) is true for nondegenerate, irreducible, reduced curves, 
and also Naito's Lemma~$3$ and Lemma~$4$ are all stated and proved for this case. 

In the proof of his Theorem~$1$, Naito uses Castelnuovo's bound for the geometric genus~$g$ of a 
nondegenerate, irreducible, smooth curve $C \subset \PP^3$:
\[
	g \le 
	\begin{cases}
		\frac{1}{4}d^2-d+1     & \text{if $d$ is even,} \\
		\frac{1}{4}(d^2-1)-d+1 & \text{if $d$ is odd.}
	\end{cases}
\]
The same bound is true for the arithmetic genus~$g_a$ of a nondegenerate, irreducible reduced curve $C \subset \PP^3$.
This can be shown by using generic initial ideals. 
For such a proof see \cite[Theorem~2.24 (Halphen's Bound)]{DeckerSchreyer}. 
Castelnuovo's bound above is a special case of Halphen's bound for $s=2$.

This leaves the cases where $(d,g)$ is either $(5,0)$, $(6,0)$, or~$(6,2)$. 
If the arithmetic genus of an irreducible curve is~$0$, then the curve is smooth. 
Therefore Naito's proof goes through in those cases. 
This leaves the case~$(6,2)$, for which the previous arguments leave the following two possibilities:
\[
\begin{matrix}
	1 & - & - & - \\
	- & - & - & - \\
	- & 3 & 1 & - \\
	- & 1 & 3 & 1 \\
\end{matrix}
\quad\quad\quad\quad
\begin{matrix}
	1 & - & - & - \\
	- & - & - & - \\
	- & 3 & - & - \\
	- & - & 3 & 1 \\
\end{matrix} \,.
\]
If $C$ is smooth, the second case is not possible since, as pointed out by Naito, 
by another theorem of Castelnuovo such a curve must have at least one $4$-secant,
hence the curve cannot be cut out by cubics. 
For $C$ only irreducible and reduced we can argue as follows:

\begin{proposition}
\label{pNoMinimal}
There exists no irreducible reduced curve $C \subset \PP^3$ with minimal free resolution
\[
	\sO_C \leftarrow \sOPP \xleftarrow{\phi_1} 3\sOPP(-3) \xleftarrow{\phi_2} 3\sOPP(-5) \xleftarrow{\phi_3} \sOPP(-6) \leftarrow 0 \,.
\]
\end{proposition}
\begin{proof}
Assume, to the contrary, that such a curve exists.
Since $C$ is locally Cohen-Macaulay, the kernel $\sE$ of~$\phi_1$ is a vector bundle 
(this is a well-known fact, nevertheless we give a proof below). 
The vector bundle~$\sE$ is also the cokernel of~$\phi_3$: 
\[
	0 \leftarrow \sE \leftarrow 3\sOPP(-5) \xleftarrow{\phi_3} \sOPP(-6) \leftarrow 0 \,.
\]
Now, $\phi_3$ is represented by a $3 \times 1$-matrix of linear forms. Let $P \in \PP^3$ be a point where
these linear forms vanish simultaneously. On the one hand, since the sequence above is a sequence of
vector bundles, its restriction to~$P$ is still exact:
\[
	0 \leftarrow \sE|_P \leftarrow 3\sO_P(-5) \xleftarrow{\phi_3|_P} \sO_P(-6) \leftarrow 0 \,.
\]
On the other hand, we have $\phi_3|_P = 0$, a contradiction.
\end{proof}

We conclude by showing the result mentioned in the proof of Proposition~\ref{pNoMinimal}.

\begin{proposition}
Let $C \subset \PP^n$ be a locally Cohen-Macaulay scheme of codimension~$2$ with a free resolution
\[
	0 \from \sO_C \from \OPn \from F_1 \from \dots \from F_k \from 0
\]
over a field $K$.
Let $\sE$ be the first syzygy sheaf of~$C$, i.e., the sheaf that makes the sequence
\[
	0 \from \sO_C \from \OPn \from F_1 \from \sE \from 0
\]
exact. Then $\sE$ is locally free. 
\end{proposition}

\begin{proof}
We show that for every $x \in \PP^n$ the stalk~$\sE_x$ is a free $\sO_{\PP^n,x}$-module. 
Since $C$ is locally Cohen-Macaulay of codimension~$2$, 
by the Hilbert-Burch theorem there exists a free resolution of~$\sO_{C,x}$ of the form
\[
	0 \from \sO_{C,x} \from \sO_{\PP^n,x} \from \sO_{\PP^n,x}^r \from \sO_{\PP^n,x}^{r-1} \from 0 \,.
\]
By Schanuel's Lemma\footnote{We thank Dario Portelli for suggesting us the use of Schanuel's Lemma.} 
we then have that $\sE_x \oplus \sO_{\PP^n,x}^r \cong \sO_{\PP^n,x}^{r-1} \oplus F_{1,x}$. 
Hence $\sE_x$ is a projective $\sO_{\PP^n,x}$-module, because it is a direct summand of a free module, 
so it is free, because projective modules over local rings are free.
\end{proof}

\section{Space curves lying on 3 independent cubics}
\label{sThreeCubics}

In this section we use the theory of generic initial ideals to bound 
the degree of space curves that lie on~$3$ independent cubics. 
We take what we need of this theory from the excellent review of~\cite{DeckerSchreyer}.

In what follows, let $C \subset \PP^3$ be an irreducible nondegenerate space curve 
and $\Gamma \subset \PP^2$ a general hyperplane section of~$C$. 

\begin{definition}
\label{dGin}
Let $I \subset \CC[w_0,w_1,w_2]$ be a homogeneous ideal and 
$\initial(I)$ be the ideal of initial forms of~$I$ with respect to the reverse lexicographic order. 
By a theorem of Galligo (see~\cite{Galligo1974} and~\cite{GreenGin}) 
there exists a non-empty Zariski open subset $U \subset GL(3,\CC)$ 
such that $\initial\bigl(g(I)\bigr)$ is constant and Borel-fixed for all $g \in U$ 
(for a definition of Borel-fixedness, see \cite[Definition~2.2]{DeckerSchreyer}). 
In this situation we set
\[
	\gin(I) := \initial\bigl(g(I)\bigr) 
\]
for $g \in U$.
\end{definition}

\begin{proposition}
For a set~$\Gamma$ of $d$ points in~$\PP^2$ the generic initial ideal is of the form
\[
	\gin(I_\Gamma) = 
	\left\langle 
		w_0^s,
		w_0^{s-1} w_1^{\lambda_{s-1}},
		\dots,
		w_0w_1^{\lambda_1},
		w_1^{\lambda_0} 
	\right\rangle
\]
for some $s \in \NN$, with 
\[
	\lambda_0 > \dots > \lambda_{s-1} \ge 1.
\]
The scalars $\{ \lambda_i \}_{i=0}^{s-1}$ are called the \emph{GP-invariants} of~$\Gamma$.
\end{proposition}

\begin{proof}
This is the remark after Definition~2.17 of~\cite{DeckerSchreyer}. 
It follows from the Borel-fixedness property of the generic initial ideal.
\end{proof}

\begin{proposition} \label{pDegree}
With the notation above, we have
\[
	d = \sum_{i=0}^{s-1} \lambda_i
\]
\end{proposition}

\begin{proof}
This is \cite[Remark 2.18]{DeckerSchreyer}.
\end{proof}

The main restriction on the possible GP-invariants comes from the following theorem:

\begin{theorem}[Connectedness of the GP-invariants, \cite{GP77}] 
If $\Gamma$ is a general hyperplane section of an irreducible reduced space curve, then
\[
	\lambda_i -1 \ge \lambda_{i+1} \ge \lambda_i-2, \quad 0 \le i \le s-2.
\]
\end{theorem}

\noindent
We also need the following important result:

\begin{theorem}[Laudal's Lemma, \cite{Laudal77}, \cite{Strano87}, \cite{Strano88}, \cite{GreenGin}]
If the general hyperplane section $\Gamma \subset \PP^2$ of an integral curve $C \subset \PP^3$
of degree $d > s^2+1$ lies on a hypersurface of degree~$s$, then the same holds for~$C$.
\end{theorem}

With these tools we can prove our bound.

\begin{proposition}
\label{pSeptic}
Let $C \subset \PP^3$ be an integral curve of degree~$d$ satisfying
\begin{itemize}
\item $h^0\bigl(I_C(2)\bigr) = 0$
\item $h^0\bigl(I_C(3)\bigr) \ge 3$.
\end{itemize}
Then $d \le 7$. Moreover, if $d=7$ then $h^0\bigl(I_C(3)\bigr) = 3$.
\end{proposition}

\begin{proof}
Let $C \subset \PP^3$ be an integral curve satisfying the conditions of the proposition
such that $d$ takes the maximal possible value 
(such a maximum value exists since $C$ lies on a complete intersection of two cubics, and therefore $d \le 9$).

The number~$d$ is at least~$7$ since curves of degree~$7$ satisfying the conditions above exist 
(for example, the determinantal septics from Proposition~\ref{p3pfaffians}). 
Since $d \ge 7 > 2^2+1$, Laudal's Lemma implies that $\Gamma$ does not lie on any quadratic hypersurface. 
Since it does lie on at least $3$ independent cubic hypersurfaces, 
the initial ideal of~$\Gamma$ must be of the form
\[
	\gin(I_\Gamma) = \langle x_0^3, x_0^2x_1, x_0^1x_1^2,x_1^{\lambda_0} \rangle
\]
and $\lambda_2 = 1, \lambda_1 = 2$. By Gruson and Peskine's Connectedness Theorem this implies
\[
	\lambda_1 \ge \lambda_0 - 2 \iff \lambda_0 \le 4.
\]
Therefore, we get
\[
	d = \lambda_0 + \lambda_1 + \lambda_2 
	  \le 1+2+4 
	  = 7 \,.
\]
from Proposition~\ref{pDegree}. 
The same calculation also shows that $h^0\bigl(I_\Gamma(3)\bigr) = 3$ if $d=7$. 
Now
\[
	3 = h^0\bigl(I_\Gamma(3)\bigr) \le h^0\bigl(I_C(3)\bigr) \le 3
\]
hence $h^0 \bigl( I_C(3) \bigr) = 3$.
\end{proof}

\section{Finite fields methods}
\label{finite_fields}

In this Appendix, we explain the technique that allows us
to infer the existence of curves over~$\CC$ with prescribed properties
from the existence of curves with the same properties over finite fields.
The latter can be witnessed, for example, via explicit computations with a computer algebra system.

\begin{lemma}
\label{lDominant}
Let $X$ be an irreducible scheme over~$\Spec(\ZZ)$ such that
\[
	\dim X \geq \delta
\]
for some $\delta \in \NN$. 
For a prime~$p$, denote by~$X_p$ the fiber of~$X$ over~$(p)$, i.e.
\xycenter{
	X_p \ar@{ (->}[r] \ar[d]^{\xi_p}& X \ar[d]^\xi\\
	(p) \ar@{ (->}[r] & \Spec(\ZZ) \\
}
Assume that we have a point $x \in X_p \subset X$ such that $\dim T_{X_p, x} \le \delta-1$. 
Then $\xi$ is dominant and $\dim X = \delta$.
\end{lemma}
\begin{proof}
Assume, for a contradiction, that the $\xi$ is not dominant.
Then $X$ lies completely over some prime~$p'$. 
Since $X_p$ is non empty we must have $p'=p$ and $X=X_p$. 
But this is impossible, since $\dim X \geq \delta$ and $\dim X_p \leq \dim T_{X,x} = \delta-1$.
So $\xi$ is dominant. 

By the assumptions, the fiber dimension of~$\xi$ is at most~$\delta-1$ and
the dimension of the basis is $\dim \Spec(\ZZ) = 1$. 
Therefore $\delta \ge \dim X \ge \delta$.
\end{proof}

\begin{lemma}
\label{lTangentQuadrics}
Let $C \subset \PP^3$ be a smooth curve of degree~$d$, 
let $Y \subset \PP^9$ be the set of quadrics $Q \subset \PP^3$ 
that are tangent to~$C$ in~$d'$ distinct points and intersect~$C$ transverally in $2d-2d'$ further points. 
Let $Q \in Y$ and $P_1, \dotsc, P_{d'}$ be the tangency points of~$C$ and~$Q$. 
Then the tangent space~$T_{Y,Q}$ of~$Y$ at~$Q$ is
\[
	T_{Y,Q} =
	\Bigl\{
		Q' \, \colon \, Q'(P_i) = 0 \text{ for } i \in \{ 1, \dotsc, d'\} 
	\Bigr\} \,.
\]
\end{lemma}
\begin{proof}
Let $P'_i$ be the tangent vector to~$C$ at~$P_i$ that is also tangent to~$Q$ at~$P_i$, for $i \in \{1, \dotsc, d'\}$. 
Let $Q'$ be a tangent vector to~$\PP^9$ at~$Q$. 
Then $Q'$ is tangent to~$Y$ if and only if
\[
	0 = (Q + \epsilon Q')(P_i + \epsilon P'_i) 
	  = \epsilon Q'(P_i) 
	\quad 
	\text{for all } i \in \{1, \dotsc, d'\} \,,
\]
where the previous equality holds over the dual numbers~$K[\epsilon] / (\epsilon^2)$.
\end{proof}

\begin{proposition}
\label{pExistsCharZero}
Assume that, over a finite field~$\FF_p$, we have a smooth curve $C_p \subset \PP^3$ of genus~$g$ and degree~$d$, 
and a smooth quadric~$Q_p \subset \PP^3$ satisfying the following conditions:
\begin{enumerate}
 \item \label{item:tangent_transveral} $C_p$ is tangent to~$Q_p$ in $d'$ points and intersects~$Q_p$ transversally in $2d-2d'$ further points,
 \item \label{item:h1} $h^1 \bigl(N_{\PP^3/C_p}\bigr) = 0$, and
 \item \label{item:independent_conditions} the tangency points of~$C_p$ and~$Q_p$ impose independent conditions on quadrics in~$\PP^3$.
 \end{enumerate}
Then there exists a smooth curve $\overline{C} \subset \PP^3$ and 
a smooth quadric~$\overline{Q}$ over~$\CC$ also satisfying condition~\eqref{item:tangent_transveral}.
\end{proposition}
\begin{proof}
Over~$\Spec(\ZZ)$, consider the following varieties: 
\begin{itemize}
\item the Hilbert scheme~$H_{g,d}$ of smooth curves in~$\PP^3$ of degree~$d$ and genus~$g$,
\item $\PP^9 = \PP(H^0(\sO_{\PP^3}(2))$ the space of quadrics in~$\PP^3$,
\item the incidence variety~$X \subset H_{d,g} \times \PP^9$ given by
\[
	X := 
	\bigl\{
		(C, Q) \, \colon \, Q \text{ is smooth and } C \cap Q \text{ satisfies condition~\eqref{item:tangent_transveral}} 
	\bigr\} \,,
\]
which comes with the natural projection $\rho \colon X \to H_{g,d}$.
\end{itemize}
Notice that $(C_p, Q_p)$ is a point in~$X$.
Let $X'$ be a component of~$X$ that contains~$(C_p, Q_p)$. 
Let $H'$ be a component of~$H_{d,g}$ containing~$\rho(X')$, 
then $X' \subset H' \times \PP^9$. 
We consider the following diagram:
\xycenter{
	\{(C_p,Q_p)\} \ar[dr] \ar@{ (->}[r]
	& X'_{C_p} \ar[d] \ar@{ (->}[r] 
	& X'_p \ar[d]^{\rho_p} \ar@{ (->}[r]
	& X' \ar[d]^\rho\\
	& \{C_p\} \ar[dr] \ar@{ (->}[r]
	& H_p' \ar[d] \ar@{ (->}[r]
	& H' \ar[d]^{\xi} \\
	& 
	& (p) \ar@{ (->}[r]
	& \Spec(\ZZ)
}
where $X'_p$ and~$H'_p$ are the fibers of~$X'$ and~$H'$ over~$p$, respectively,
and $X'_{C_p}$ is the fiber of~$\rho_p$ over~$C_p$. 

First, we want to use Lemma~\ref{lDominant} on the square
\xycenter{
	  C_p \ar@{}[r]|-{\in}
	& H_p' \ar[d] \ar@{ (->}[r]
	& H' \ar[d]^{\xi} \\
	& (p) \ar@{ (->}[r]
	& \Spec(\ZZ)
}
to conclude that its dimension is $4d+1$. 
First, we show\footnote{We thank Barbara Fantechi for a useful discussion about Hilbert schemes.} that the dimension of~$H'$ is at least~$4d+1$. 
In fact, the curve~$C_p$ satisfies the condition of being \emph{locally unobstructed},
according to \cite[Definition 2.11 (1)]{Kollar1996}, due to \cite[Lemma 2.12]{Kollar1996}.
Hence \cite[Theorem 2.15~(1)]{Kollar1996} implies that the dimension of~$H'$ is at least $\chi(N_{\PP^3/C_p}) + 1$,
and \cite[Lemma~5]{Ein1986} shows that $\chi(N_{\PP^3/C_p}) = 4d$.
Thus $\dim H' \geq 4d+1$. Second, by our assumption $h^1 \bigl(N_{\PP^3/C_p}\bigr) = 0$, 
the tangent space of~$H'_p$ at~$C_p$ has the expected dimension $4d$. 
Therefore, by Lemma~\ref{lDominant} we get that $\xi$ is dominant and $\dim H' = 4d+1$.

Now we want to use again Lemma~\ref{lDominant} on the square:
\xycenter{
	(C_p,Q_p) \ar@{}[r]|-{\in}
	& X'_p \ar[d]^{\rho_p} \ar@{ (->}[r]
	& X' \ar[d]^\rho\\
	& H_p' \ar[d]
	& H' \ar[d]^{\xi} \\
	& (p) \ar@{ (->}[r]
	& \Spec(\ZZ)
}
Since being tangent imposes at most one condition on a pair~$(C,Q)$, we have
\[
	\dim X' \ge \dim H' + 9 - d' = 4d + 10 - d' \,.
\]
Next, we estimate the dimension of the tangent space of~$X_p'$ at~$(C_p,Q_p)$.
Certainly, this tangent space is contained in the product of the tangent space 
of the base~$H'_p$ at~$C_p$ and the tangent space of the fiber~$X'_{C_p}$ at~$(C_p,Q_p)$:
\[
	T_{X'_p,(C_p,Q_p)} \subset T_{H'_p,C_p} \times T_{X'_{C_p},(C_p,Q_p)}
\]
By our assumptions and Lemma~\ref{lTangentQuadrics} we have
\[
	\dim T_{X'_{C_p},(C_p,Q_p)} = 9 - d' \,,
\]
while 
\[
	\dim T_{H'_p,C_p} = 4d \,.
\]
So
\[
	\dim T_{X'_p,(C_p,Q_p)} \le 4d + 9 - d' \,,
\]
therefore by Lemma~\ref{lDominant} we get that $\xi \circ \rho$ is dominant. 
This implies that the fiber of~$X'$ over the generic point of~$\Spec(\ZZ)$ is not empty. 
Hence there exist a curve~$\overline{C}$ and a quadric~$\overline{Q}$ that satisfy condition (i) in the statement.
\end{proof}

\def\cprime{$'$} \def\cprime{$'$}


\begin{thebibliography}{GNSS17}

\bibitem[BE77]{BuEi}
D.~A. Buchsbaum and D.~Eisenbud.
\newblock {Algebra Structures for finite free Resolutions and some Structure
  Theorems for Ideals of Codimension 3}.
\newblock {\em Amer. J. Math.}, 99(3):447--485, 1977.
\newblock \href {http://dx.doi.org/10.2307/2373926}
  {\path{doi:10.2307/2373926}}.

\bibitem[Bor08]{Borel1908}
\'E. Borel.
\newblock M\'emoire sur les d\'eplacements \`a trajectoires sph\'eriques.
\newblock {\em M\'emoires pr\'esent\'es par divers savants}, 2(33):1--128,
  1908.

\bibitem[Bri97]{Bricard1897}
R.~Bricard.
\newblock M\'emoire sur la th\'eorie de l'octa\`edre articul\'e.
\newblock {\em Journal de Math\'ematiques pures et appliqu\'ees}, 3:113--148,
  1897.

\bibitem[Bri06]{Bricard1906}
R.~Bricard.
\newblock M\'emoire sur les d\'eplacements \`a trajectoires sph\'eriques.
\newblock {\em Journal de \'Ecole Polytechnique}, 11(2):1--96, 1906.

\bibitem[Die96]{Dietmaier1996}
P.~Dietmaier.
\newblock {\em Forward Kinematics and Mobility Criteria of One Type of
  Symmetric Stewart-Gough Platforms}, pages 379--388.
\newblock Springer Netherlands, Dordrecht, 1996.
\newblock \href {http://dx.doi.org/10.1007/978-94-009-1718-7_38}
  {\path{doi:10.1007/978-94-009-1718-7_38}}.

\bibitem[DS00]{DeckerSchreyer}
W.~Decker and F.-O. Schreyer.
\newblock Non-general type surfaces in {${\bf P}\sp 4$}: some remarks on bounds
  and constructions.
\newblock {\em J. Symbolic Comput.}, 29(4-5):545--582, 2000.
\newblock Symbolic computation in algebra, analysis, and geometry (Berkeley,
  CA, 1998).
\newblock \href {http://dx.doi.org/10.1006/jsco.1999.0323}
  {\path{doi:10.1006/jsco.1999.0323}}.

\bibitem[Ein86]{Ein1986}
L.~Ein.
\newblock Hilbert scheme of smooth space curves.
\newblock {\em Ann. Sci. \'{E}cole Norm. Sup. (4)}, 19(4):469--478, 1986.
\newblock \href {http://dx.doi.org/10.24033/asens.1513}
  {\path{doi:10.24033/asens.1513}}.

\bibitem[Eis05]{Eisenbud2005}
D.~Eisenbud.
\newblock {\em {The geometry of syzygies. A second course in commutative
  algebra and algebraic geometry}}, volume 229 of {\em Graduate Texts in
  Mathematics}.
\newblock Springer-Verlag, New York, 2005.
\newblock \href {http://dx.doi.org/10.1007/b137572}
  {\path{doi:10.1007/b137572}}.

\bibitem[Gal74]{Galligo1974}
A.~Galligo.
\newblock {A propos du th{\'e}or{\`e}me de pr{\'e}paration de Weierstrass}.
\newblock In F.~Norguet, editor, {\em {Fonctions de Plusieurs Variables
  Complexes}}, pages 543--579, Berlin, Heidelberg, 1974. Springer.
\newblock \href {http://dx.doi.org/10.1007/BFb0068121}
  {\path{doi:10.1007/BFb0068121}}.

\bibitem[GNS15]{BondTheory}
M.~Gallet, G.~Nawratil, and J.~Schicho.
\newblock Bond theory for pentapods and hexapods.
\newblock {\em J. Geom.}, 106(2):211--228, 2015.
\newblock \href {http://dx.doi.org/10.1007/s00022-014-0243-1}
  {\path{doi:10.1007/s00022-014-0243-1}}.

\bibitem[GNS17]{Gallet2017}
M.~Gallet, G.~Nawratil, and J.~Schicho.
\newblock Liaison linkages.
\newblock {\em J. Symbolic Comput.}, 79(part 1):65--98, 2017.
\newblock \href {http://dx.doi.org/10.1016/j.jsc.2016.08.006}
  {\path{doi:10.1016/j.jsc.2016.08.006}}.

\bibitem[GNSS17]{GalletNawratilSchichoSelig}
M.~Gallet, G.~Nawratil, J.~Schicho, and J.~M. Selig.
\newblock Mobile icosapods.
\newblock {\em Adv. in Appl. Math.}, 88:1--25, 2017.
\newblock \href {http://dx.doi.org/10.1016/j.aam.2016.12.002}
  {\path{doi:10.1016/j.aam.2016.12.002}}.

\bibitem[GP78]{GP77}
L.~Gruson and C.~Peskine.
\newblock Genre des courbes de l'espace projectif.
\newblock In {\em Algebraic geometry ({P}roc. {S}ympos., {U}niv. {T}roms\o ,
  {T}roms\o , 1977)}, volume 687 of {\em Lecture Notes in Math.}, pages 31--59.
  Springer, Berlin, 1978.
\newblock \href {http://dx.doi.org/10.1007/BFb0062927}
  {\path{doi:10.1007/BFb0062927}}.

\bibitem[Gre98]{GreenGin}
M.~L. Green.
\newblock Generic initial ideals.
\newblock In {\em Six lectures on commutative algebra (Bellaterra, 1996)},
  volume 166 of {\em Progr. Math.}, pages 119--186. Birkh{\"a}user, Basel,
  1998.
\newblock \href {http://dx.doi.org/10.1007/978-3-0346-0329-4_2}
  {\path{doi:10.1007/978-3-0346-0329-4_2}}.

\bibitem[GS09]{Geis2009}
F.~Gei{\ss} and F.-O. Schreyer.
\newblock A family of exceptional {S}tewart-{G}ough mechanisms of genus $7$.
\newblock In {\em Interactions of classical and numerical algebraic geometry},
  volume 496, pages 221--234. American Mathematical Society, 2009.
\newblock \href {http://dx.doi.org/10.1090/conm/496/09725}
  {\path{doi:10.1090/conm/496/09725}}.

\bibitem[GvB20]{software}
H.-C. Graf~von Bothmer.
\newblock Software code for the paper {``Hexapods with a small linear span''},
  2020.
\newblock \href {http://dx.doi.org/10.5281/zenodo.4309767}
  {\path{doi:10.5281/zenodo.4309767}}.

\bibitem[Har10]{Hartshorne2010}
R.~Hartshorne.
\newblock {\em Deformation theory}, volume 257 of {\em Graduate Texts in
  Mathematics}.
\newblock Springer, New York, 2010.
\newblock \href {http://dx.doi.org/10.1007/978-1-4419-1596-2}
  {\path{doi:10.1007/978-1-4419-1596-2}}.

\bibitem[HK00]{Husty2000a}
M.~L. Husty and A.~Karger.
\newblock {A}rchitecture {S}ingular {P}arallel {M}anipulators and their
  {S}elf-{M}otions.
\newblock In J.~Lenar{\v{c}}i{\v{c}} and M.~M. Stani{\v{s}}i{\'{c}}, editors,
  {\em Advances in Robot Kinematics}, pages 355--364. Springer Netherlands,
  Dordrecht, 2000.
\newblock \href {http://dx.doi.org/10.1007/978-94-011-4120-8_37}
  {\path{doi:10.1007/978-94-011-4120-8_37}}.

\bibitem[HSW18]{Hauenstein2018}
J.~D. Hauenstein, S.~N. Sherman, and C.~W. Wampler.
\newblock {Exceptional Stewart-Gough platforms, Segre embeddings, and the
  special Euclidean group}.
\newblock {\em SIAM J. Appl. Algebra Geom.}, 2(1):179--205, 2018.
\newblock \href {http://dx.doi.org/10.1137/17M1114284}
  {\path{doi:10.1137/17M1114284}}.

\bibitem[Kar08]{Karger2008}
A.~Karger.
\newblock Architecturally singular non-planar parallel manipulators.
\newblock {\em Mech. Mach. Theory}, 43(3):335--346, 2008.
\newblock \href {http://dx.doi.org/10.1007/978-94-015-9064-8_45}
  {\path{doi:10.1007/978-94-015-9064-8_45}}.

\bibitem[Kol96]{Kollar1996}
J.~Koll\'{a}r.
\newblock {\em Rational curves on algebraic varieties}, volume~32 of {\em
  Ergebnisse der Mathematik und ihrer Grenzgebiete. 3. Folge. A Series of
  Modern Surveys in Mathematics [Results in Mathematics and Related Areas. 3rd
  Series. A Series of Modern Surveys in Mathematics]}.
\newblock Springer-Verlag, Berlin, 1996.
\newblock \href {http://dx.doi.org/10.1007/978-3-662-03276-3}
  {\path{doi:10.1007/978-3-662-03276-3}}.

\bibitem[Lau78]{Laudal77}
O.~A. Laudal.
\newblock A generalized trisecant lemma.
\newblock In {\em Algebraic geometry ({P}roc. {S}ympos., {U}niv. {T}roms\o ,
  {T}roms\o , 1977)}, volume 687 of {\em Lecture Notes in Math.}, pages
  112--149. Springer, Berlin, 1978.
\newblock \href {http://dx.doi.org/10.1007/BFb0062930}
  {\path{doi:10.1007/BFb0062930}}.

\bibitem[Mou96]{Mourrain1996}
B.~Mourrain.
\newblock Enumeration problems in geometry, robotics and vision.
\newblock In {\em Algorithms in algebraic geometry and applications
  ({S}antander, 1994)}, volume 143 of {\em Progr. Math.}, pages 285--306.
  Birkh\"auser, Basel, 1996.
\newblock \href {http://dx.doi.org/10.1007/978-3-0348-9104-2_14}
  {\path{doi:10.1007/978-3-0348-9104-2_14}}.

\bibitem[Nai02]{deg6resolutions}
H.~Naito.
\newblock Minimal free resolution of curves of degree 6 or lower in the
  3-dimensional projective space.
\newblock {\em Tokyo J. Math.}, 25(1):191--196, 2002.
\newblock \href {http://dx.doi.org/10.3836/tjm/1244208945}
  {\path{doi:10.3836/tjm/1244208945}}.

\bibitem[Naw09]{Nawratil2009}
G.~Nawratil.
\newblock A new approach to the classification of architecturally singular
  parallel manipulators.
\newblock In {\em Proceedings of the 5th International Workshop on
  Computational Kinematics}, pages 349--358. Springer Berlin Heidelberg, 2009.
\newblock \href {http://dx.doi.org/10.1007/978-3-642-01947-0_43}
  {\path{doi:10.1007/978-3-642-01947-0_43}}.

\bibitem[Naw12]{Nawratil2012}
G.~Nawratil.
\newblock {S}elf-{M}otions of {P}lanar {P}rojective {S}tewart {G}ough
  {P}latforms.
\newblock In Jadran Lenarcic and Manfred Husty, editors, {\em Latest {A}dvances
  in {R}obot {K}inematics}, pages 27--34. Springer Netherlands, Dordrecht,
  2012.
\newblock \href {http://dx.doi.org/10.1007/978-94-007-4620-6_4}
  {\path{doi:10.1007/978-94-007-4620-6_4}}.

\bibitem[Naw13]{Nawratil2013c}
G.~Nawratil.
\newblock On equiform {S}tewart {G}ough platforms with self-motions.
\newblock {\em Journal for Geometry and Graphics}, 17(2):163--175, 2013.

\bibitem[Naw14a]{Nawratil2014c}
G.~Nawratil.
\newblock Congruent {S}tewart {G}ough platforms with non-translational
  self-motions.
\newblock In {\em Proceedings of 16th International Conference on Geometry and
  Graphics}, pages 204--215, August 2014.

\bibitem[Naw14b]{NawratilIntro}
G.~Nawratil.
\newblock {Introducing the theory of bonds for Stewart Gough platforms with
  selfmotions}.
\newblock {\em J. Mechanisms Robotics}, 6(1):49--57, 2014.
\newblock \href {http://dx.doi.org/10.1115/1.4025623}
  {\path{doi:10.1115/1.4025623}}.

\bibitem[Naw14c]{Nawratil2014a}
G.~Nawratil.
\newblock On the {S}elf-{M}obility of {P}oint-{S}ymmetric {H}exapods.
\newblock {\em Symmetry}, 6(4):954, 2014.
\newblock \href {http://dx.doi.org/10.3390/sym6040954}
  {\path{doi:10.3390/sym6040954}}.

\bibitem[Naw18]{Nawratil2018}
G.~Nawratil.
\newblock Hexapods with plane-symmetric self-motions.
\newblock {\em Robotics}, 7(2), 2018.
\newblock \href {http://dx.doi.org/10.3390/robotics7020027}
  {\path{doi:10.3390/robotics7020027}}.

\bibitem[RM98]{Roschel1998}
O.~R{\"o}schel and S.~Mick.
\newblock Characterisation of {A}rchitecturally {S}haky {P}latforms.
\newblock In {\em Advances in Robot Kinematics: Analysis and Control}, pages
  465--474. Springer Netherlands, Dordrecht, 1998.
\newblock \href {http://dx.doi.org/10.1007/978-94-015-9064-8_47}
  {\path{doi:10.1007/978-94-015-9064-8_47}}.

\bibitem[Str87]{Strano87}
R.~Strano.
\newblock Sulle sezioni iperpiane delle curve.
\newblock {\em Rendiconti del Seminario Matematico e Fisico di Milano},
  57(1):125--134, Dec 1987.
\newblock \href {http://dx.doi.org/10.1007/BF02925046}
  {\path{doi:10.1007/BF02925046}}.

\bibitem[Str88]{Strano88}
R.~Strano.
\newblock A characterization of complete intersection curves in {${\bf P}^3$}.
\newblock {\em Proc. Amer. Math. Soc.}, 104(3):711--715, 1988.
\newblock \href {http://dx.doi.org/10.2307/2046779}
  {\path{doi:10.2307/2046779}}.

\end{thebibliography}
\end{document}